\documentclass[reqno, 12pt]{amsart}

\usepackage{amssymb,amsfonts, latexsym,amsmath,hhline,array,longtable,enumitem}
\usepackage{pdfsync,color, comment,colortbl, mathrsfs,stmaryrd,cite,graphicx,lscape}
\usepackage{caption, subcaption}
\usepackage{ifpdf}
\ifpdf 
\usepackage[colorlinks=true, citecolor=blue, linkcolor=blue, urlcolor=blue]{hyperref} 
\fi

\newtheorem{formula}{}[section]
\newtheorem{definition}[formula]{Definition}

\newtheorem{theorem}[formula]{Theorem}

\newtheorem{example}[formula]{Example}

\usepackage{stmaryrd}

\renewcommand{\mod}{\text{mod} \;}

\usepackage{xcolor}

\begin{document}

\title{Switching graphs and Hadamard
matrices}
\author{Aida Abiad}
\address{Department of Mathematics and Computer Science, 
Eindhoven University of Technology, The Netherlands}
\address{Department of Mathematics and Data Science of Vrije Universiteit Brussels, Belgium}
\email{a.abiad.monge@tue.nl}
\author{Louka Peters}
\address{Department of Mathematics: Analysis, Logic and Discrete Mathematics, Ghent University, Belgium}
\email{louka.peters@ugent.be}

\date{}

\begin{abstract}
Local operations of combinatorial structures (graphs, Hadamard matrices, codes, designs) that maintain the basic parameters unaltered, have been widely used in the literature under the name of switching. We show an equivalence between two switching methods to construct inequivalent Hadamard matrices, which were proposed by Orrick [\emph{SIAM Journal on Discrete Mathematics}, 2008], and the switching method for constructing cospectral graphs which was introduced by Godsil and McKay [\emph{Aequationes Mathematicae}, 1982].\\


\noindent \textbf{MCS:} 05C50
\end{abstract}

\maketitle

\section{Introduction}

Multiple local operations of combinatorial structures (such as \linebreak Hadamard matrices, graphs, codes or designs) that leave the basic parameters unaltered, have been widely used in the literature under the name of \emph{switching}. For such switching methods to work, the combinatorial object has to satisfy some conditions.

One of the multiple known methods to construct Hadamard matrices is due to Orrick \cite{Orrick2008}, who proposed two operations (switching a Hall set and switching a closed quadruple) that interchange substructures of Hadamard matrices, thereby producing new, generally inequivalent, Hadamard matrices. Orrick's switching methods have been used for the enumeration of different equivalence classes of Hadamard matrices, $D$-optimal designs \cite{brent2011finding,orrick2008enumeration}, or Hadamard codes \cite{bloomquist2013classifying}, among others.

On the other hand, switching operations for the construction of cospectral graphs (graphs with the same spectrum) have been introduced in the literature \cite{Godsil1982,Abiad2012,wang2019cospectral, qiu2020theorem,ABS2024switching}. Obtaining cospectral graphs is useful for understanding what graph properties cannot be detected by the spectrum, therefore providing new insights to a conjecture in this area due to Haemers (``almost all graphs are determined by their spectrum''). Godsil and McKay \cite{Godsil1982} introduced one of the first and most fruitful switching methods (GM-switching), which has witnessed several applications. Instances of it are the construction of new strongly regular graphs, or its use for providing a negative answer to Haemers' conjecture \cite{which} for certain graph classes (see e.g.\ \cite{abiad2016switched,cioba,Kubota2016,johnson,haemers2010graphs}).

The link between switching methods for combinatorial designs and codes was investigated in \cite{O2012}. In the same paper the following is mentioned: ``Graphs are not in general considered in the current work; graph switching has been discussed and applied in a variety
of places, and this topic desires a treatment of its own.'' Here, we focus in this problem, and we show  a relation between switching methods for Hadamard matrices and graphs. In particular, we prove an equivalence between  methods to construct Hadamard matrices (switching a Hall set and switching a closed quadruple) by Orrick \cite{Orrick2008} and the switching method to construct cospectral graphs by Godsil and McKay \cite{Godsil1982}. To do so, we use Hadamard graphs, which were introduced by McKay in \cite{McKay1979}.

This paper is organized as follows. Section \ref{sec:prelim} presents the needed notation and preliminary results. Section \ref{sec:newresults} contains the main results (Theorems \ref{closedGM}-\ref{thm:converseGMHallset}), which  establish the equivalences between the switching methods for graphs and Hadamard matrices. Section \ref{sec:concremarks} ends with some concluding remarks and open problems.

\section{Preliminaries}\label{sec:prelim}

Let us start by giving some background information on the known switching methods to construct cospectral graphs and Hadamard matrices.

\subsection{Switching to construct cospectral graphs}

We consider simple graphs. 
The \emph{(adjacency) spectrum} of a graph is the multiset of eigenvalues of its adjacency matrix. Graphs are \emph{cospectral} if they have the same spectrum. Two graphs are said to be \emph{cospectral mates} if they are cospectral and non-isomorphic. Let \(I\) denote the identity matrix, \(J\) the all-one matrix and \(j_n\) the all-one vector of length \(n\).

\begin{theorem}[GM-switching \cite{Godsil1982}]\label{GMswitch}
    Let $G$ be a graph and consider \(\{C_1,\dots,C_t,D\}\) to be a partition of its vertices such that, for all \(i,j\in\{1,\dots,t\}\):
    \begin{enumerate}
        \item[(i)] Every vertex in \(C_i\) has the same number of neighbours in \(C_j\).
        \item[(ii)] Every vertex in \(D\) has \(0\), \(\frac{1}{2}|C_i|\) or \(|C_i|\) neighbours in \(C_i\). 
    \end{enumerate}
    For all \(i\in\{1,\dots,t\}\) and every \(v\in D\) that has exactly \(\frac{1}{2}|C_i|\) neighbours in \(C_i\), swap the adjacencies between \(v\) and \(C_i\). The resulting graph is cospectral with $G$.
\end{theorem}

The partition \(\{C_1,\dots,C_t,D\}\) in Theorem \ref{GMswitch} is called the \emph{switching partition} and the set \(\{C_1,\dots,C_t\}\) is the \emph{switching set}. The simplest nontrivial case of Theorem \ref{GMswitch} has one switching block of size four. This case has actually been the most fruitful in the literature, see e.g.\ \cite{abiad2019graph,johnson,abiad2016switched,haemers2010graphs}. 

\subsection{Switching to construct Hadamard matrices}\label{Hadamard}

A \emph{Hadamard matrix} \(H\) is a square \(n \times n\) matrix whose elements are all 1 or \(-1\) and that satisfies \(HH^T = nI\). It is well known that Hadamard matrices only exist for orders 1, 2 and multiples of 4. Two Hadamard matrices are \emph{equivalent} if one can be formed into the other by permutations of any rows and/or columns, and negation of any rows and/or columns. When a Hadamard matrix has a constant row and column sum, it is \emph{regular}.
The smallest Hadamard matrices (up to equivalence) are given below, where we denote \(-1\) as \(-\).
\begin{equation*}\label{smallHad}
\begin{pmatrix}
    1
\end{pmatrix}
\qquad
\begin{pmatrix}
    1&1\\
    1&-
\end{pmatrix}
\qquad
\begin{pmatrix}
    1&1&1&1\\
    1&1&-&-\\
    1&-&1&-\\
    1&-&-&1\\
\end{pmatrix}
\end{equation*}


In the literature, different kinds of normalization on Hadamard matrices are defined. A Hadamard matrix is often called \emph{normalized} if its first row and column are all positive, for example, see \cite{Hedayat1978} and \cite{Jayathilake2013}. 
The 3-normalisation was first introduced by Orrick and Solomon in \cite{Orrick2007} and later redefined by Orrick in \cite{Orrick2008} to be more general. The latter definition is given here.
Denote \(h_i\) as the \(i^{th}\) row of the matrix and \(h_{ij}\) as its \(j^{th}\) element.
The \emph{Hadamard product} is then defined by \[h_i \circ h_j = (h_{i1}h_{j1},\, \dots,\, h_{in}h_{jn}).\]
\begin{definition}
    A Hadamard matrix is called \emph{3-normalized} (on rows \((i,j,k)\)), if there exist three rows \(h_i, h_j, h_k\) for which \(h_i \circ h_j \circ h_k = j_n\), with \(j_n\) the row vector consisting solely of ones. 
\end{definition}
A standard form of a 3-normalized matrix is given by \eqref{3norm}:
\begin{equation} \label{3norm}
\begin{pmatrix}
1\cdots1 & -\cdots- & -\cdots- & 1\cdots1\\
1\cdots1 & -\cdots- & 1\cdots1 & -\cdots-\\
1\cdots1 & 1\cdots1 & -\cdots- & -\cdots-\\
a_4 & b_4 & c_4 & d_4\\
\vdots&\vdots&\vdots&\vdots\\
a_n& b_n & c_n & d_n\\
\end{pmatrix}\end{equation}
with \(a_i, b_i, c_i, d_i\) being \(1\times \frac{n}{4}\) dimensional \(\{1,-1\}\)-vectors, \(i \in \{4,\dots,n\}\).

When only considering the submatrix consisting of these three rows, each column has an even amount of \(-1\)'s, thus are of the form \linebreak \((1,1,1), (1,-1,-1), (-1,1,-1)\) or \((-1,-1,1)\). We can now define a partition on the columns of \(H\) into four equal parts, based upon these options. This partition is called a \emph{field structure}, and the different classes are the \emph{fields}, all of length \(\frac{n}{4}\). 
Remarkable about these fields is that when considering a row \(r \notin \{i,j,k\}\), the sum of the elements in one field is the same for all four fields. This can be shown as follows. Consider the Hadamard matrix to be of the form \eqref{3norm}. If \(a_r\) (resp. \(b_r, c_r, d_r\)) is the sum of the elements of row \(r \in \{4,\dots, n\}\) in the first (resp. second, third and fourth) field, we have \begin{align*}
    a_r - b_r - c_r + d_r &= 0\\
    a_r -b_r + c_r - d_r &=0\\
    a_r + b_r - c_r - d_r &=0
\end{align*}
because of orthogonality with the 3-normalized rows.
It follows that \(a_r = b_r = c_r = d_r\). Since the sum of the elements in a field is even when \(\frac{n}{4}\) is even and odd when \(\frac{n}{4}\) is odd, \(a_r \equiv \frac{n}{4} \; (\mod 2)\), so the total sum of the row is congruent to \(n \; (\mod 8)\). 

How many inequivalent Hadamard matrices there are of a given order, is for many orders still an open question. To construct new inequivalent Hadamard matrices from existing ones,  switching methods have been introduced in the literature. In this work we focus on two such switching methods described by Orrick \cite{Orrick2008}, specifically for 3-normalized matrices:
\begin{itemize}
    \item switching a closed quadruple,
    \item switching a Hall set. 
\end{itemize}

\subsubsection{Switching a closed quadruple}

Recall the \emph{Hadamard product}:

$$h_i \circ h_j = (h_{i1}h_{j1},\, \dots,\, h_{in}h_{jn}).$$
\begin{definition}
    A \emph{closed quadruple} is a set of four rows \((i,j,k,l)\), where \(h_i \circ h_j \circ h_k \circ h_l= \pm j_n\).
\end{definition}

It is readily seen that a 3-normalized Hadamard matrix which also contains \(\pm j_n\) as a separate row, has a closed quadruple. A standard form of such a Hadamard matrix is as follows:

\begin{equation} \label{closedquad}
\begin{pmatrix}
1\cdots1 & -\cdots- & -\cdots- & 1\cdots1\\
1\cdots1 & -\cdots- & 1\cdots1 & -\cdots-\\
1\cdots1 & 1\cdots1 & -\cdots- & -\cdots-\\
1\cdots1 & 1\cdots1 & 1\cdots1 & 1\cdots1\\
a_5 & b_5 & c_5 & d_5\\
\vdots&\vdots&\vdots&\vdots\\
a_n& b_n & c_n & d_n\\
\end{pmatrix}\end{equation}
with \(a_i, b_i, c_i, d_i\) being \(1\times \frac{n}{4}\) dimensional \(\{1,-1\}\)-vectors.

\begin{definition} Let \(H\) be a Hadamard matrix that contains a closed quadruple \(Q\).  Denote \(Q_i\) as all the elements of \(Q\) in the field \(i\). \emph{Switching a closed quadruple} means negating \(Q_i\) for some \(i\in \{1,2,3,4\}\).
\end{definition}

\begin{theorem}\textbf{(Switching a closed quadruple)} \cite{Orrick2008}
    The matrix \(H'\) obtained by switching a closed quadruple of a Hadamard matrix \(H\) is again a Hadamard matrix.
\end{theorem}

\subsubsection{Switching a Hall set}

\begin{definition}
    A \emph{Hall set} of a Hadamard matrix is a 3-normalized quadruple of rows (or columns), where the fourth row (or column) has one odd-sign entry in each field. 
\end{definition}
A general form of a Hall set is given by \eqref{Hallset}:
\begin{equation}\label{Hallset}
    \begin{pmatrix}
1\cdots1 \; 1& -\cdots- \ -& -\cdots- \ -& 1\cdots1 \; 1\\
1\cdots1 \; 1& -\cdots- \ -& 1\cdots1 \; 1& -\cdots- \ -\\
1\cdots1 \; 1& 1\cdots1 \; 1& -\cdots- \ -& -\cdots- \ -\\
1\cdots1 \, -& 1\cdots1 \, -& 1\cdots1 \ -& 1\cdots1 \; -\\
\end{pmatrix}.
\end{equation}
The Hadamard product of the Hall set thus delivers a row vector with exactly four elements of the opposite sign. The four corresponding columns are called the \emph{Hall columns}. 

\begin{theorem}\cite{Orrick2008}
    Let \(H\) be a Hadamard matrix of size \(n\) with a Hall set. If \(n \equiv 0 \; (\mod 8)\) then the Hall columns form a closed quadruple. If \(n \equiv 4 \;(\mod 8)\) then the Hall columns form a Hall set.
\end{theorem}

From now on we only consider \(n \equiv 4 \; (\mod 8)\) when talking about Hall sets, and will consider the `Hall set' to be the Hall set rows and the Hall set columns together. Hadamard matrices with a Hall set are equivalent to \eqref{Hallsetmatrix}, where all \(A_{ij}\) have row and column sums equal to two if \(i = j\) and 0 otherwise: 
\begin{equation} \label{Hallsetmatrix}
\begin{pmatrix}
    1&-&-&-&1\cdots1&1\cdots1&1\cdots1&-\cdots-\\
    -&1&-&-&1\cdots1&-\cdots-&-\cdots-&-\cdots-\\
    -&-&1&-&1\cdots1&-\cdots-&1\cdots1&1\cdots1\\
    -&-&-&1&1\cdots1&1\cdots1&-\cdots-&1\cdots1\\
    1&1&1&1&&&&\\
    \vdots&\vdots&\vdots&\vdots&A_{11}&A_{12}&A_{13}&A_{14}\\
    1&1&1&1&&&&\\
    -&1&1&-&&&&\\
    \vdots&\vdots&\vdots&\vdots&A_{21}&A_{22}&A_{23}&A_{24}\\
    -&1&1&-&&&&\\
    -&1&-&1&&&&\\
    \vdots&\vdots&\vdots&\vdots&A_{31}&A_{32}&A_{33}&A_{34}\\
    -&1&-&1&&&&\\
    1&1&-&-&&&&\\
    \vdots&\vdots&\vdots&\vdots&A_{41}&A_{42}&A_{43}&A_{44}\\
    1&1&-&-&&&&\\
\end{pmatrix}.
\end{equation}
Denote the different blocks of \eqref{Hallsetmatrix} as follows: 
\[H_4 = 
\begin{pmatrix}
    1&-&-&-\\
    -&1&-&-\\
    -&-&1&-\\
    -&-&-&1\\
\end{pmatrix},\quad
F_1 = \begin{pmatrix}
    1&\cdots&1\\
    1&\cdots&1\\
    1&\cdots&1\\
    1&\cdots&1\\
\end{pmatrix},\quad
F_2 = \begin{pmatrix}
    1&\cdots&1\\
    -&\cdots&-\\
    -&\cdots&-\\
    1&\cdots&1\\
\end{pmatrix},
\]
\[
F_3 = \begin{pmatrix}
    1&\cdots&1\\
    -&\cdots&-\\
    1&\cdots&1\\
    -&\cdots&-\\
\end{pmatrix},
\quad
F_4 = \begin{pmatrix}
    -&\cdots&-\\
    -&\cdots&-\\
    1&\cdots&1\\
    1&\cdots&1\\
\end{pmatrix},
\] 
\[G_1 = F_1^T, \; G_j = -F_j^T \text{ for } j \in \{2,3,4\},\]
resulting in 
\begin{equation}
    \begin{pmatrix}
        H_4 & F_1 & F_2 & F_3 & F_4\\
        G_1 &A_{11}&A_{12}&A_{13}&A_{14}\\ 
        G_2 &A_{21}&A_{22}&A_{23}&A_{24}\\
        G_3& A_{31}&A_{32}&A_{33}&A_{34}\\
        G_4 &A_{41}&A_{42}&A_{43}&A_{44}\\
    \end{pmatrix}.
\end{equation}

\begin{definition}
    \emph{Switching a Hall set} means negating \(F_i\) and corresponding \(G_i\) for one \(i \in \{1,2,3,4\}\). 
\end{definition}

\begin{theorem}\textbf{(Switching a Hall set)} \cite{Orrick2008}
    The matrix produced by switching a Hall set in a Hadamard matrix is again a Hadamard matrix.
\end{theorem}

For example, when switching a Hall set of the matrix \(H\), shown in \eqref{H12}, with \(i=1\), we get matrix \(H'\) in \eqref{H12'}. In this example the matrix remains equivalent after switching:
\begin{equation}\label{H12}
H = \begin{pmatrix}
    1&-&-&-&1&1&1&1&1&1&-&-\\
    -&1&-&-&1&1&-&-&-&-&-&-\\
    -&-&1&-&1&1&-&-&1&1&1&1\\
    -&-&-&1&1&1&1&1&-&-&1&1\\
    1&1&1&1&1&1&-&1&1&-&-&1\\
    1&1&1&1&1&1&1&-&-&1&1&-\\
    -&1&1&-&-&1&1&1&1&-&1&-\\
    -&1&1&-&1&-&1&1&-&1&-&1\\
    -&1&-&1&1&-&-&1&1&1&1&-\\
    -&1&-&1&-&1&1&-&1&1&-&1\\
    1&1&-&-&-&1&-&1&-&1&1&1\\
    1&1&-&-&1&-&1&-&1&-&1&1\\
    \end{pmatrix},
\end{equation}
\begin{equation}\label{H12'}
    H'=
    \begin{pmatrix}
    1&-&-&-&-&-&1&1&1&1&-&-\\
    -&1&-&-&-&-&-&-&-&-&-&-\\
    -&-&1&-&-&-&-&-&1&1&1&1\\
    -&-&-&1&-&-&1&1&-&-&1&1\\
    -&-&-&-&1&1&-&1&1&-&-&1\\
    -&-&-&-&1&1&1&-&-&1&1&-\\
    -&1&1&-&-&1&1&1&1&-&1&-\\
    -&1&1&-&1&-&1&1&-&1&-&1\\
    -&1&-&1&1&-&-&1&1&1&1&-\\
    -&1&-&1&-&1&1&-&1&1&-&1\\
    1&1&-&-&-&1&-&1&-&1&1&1\\
    1&1&-&-&1&-&1&-&1&-&1&1\\
    \end{pmatrix}.\end{equation}

\subsection{Graphs from Hadamard matrices: Hadamard graphs}\label{graphhad}

It is possible to create a graph from a Hadamard matrix. Different construction methods for this are known, see e.g. \cite{Jayathilake2013, McKay1979}. Next we briefly discuss one of them: Hadamard graphs.

McKay \cite{McKay1979} introduced a way for constructing (non-simple) graphs from any \(\{1, -1\}\)-matrix (not necessarily Hadamard matrices). This is the method that we will use in Section \ref{sec:newresults} to show the new equivalences between switching on graphs and switching on Hadamard matrices.

\begin{definition}\label{HadamardGraphdef}
Let \(H\) be an \(m \times n\) \(\{1,-1\}\)-matrix. Define \(G_H\) as the graph with vertices \(\{v_1, \dots, v_m, v'_1, \dots, v'_m, w_1, \dots w_n, w'_1, \dots w'_n\}\) and edges 
\[
\begin{cases}
    (v_i,w_j),\text{ and } (v'_i, w'_j) &\quad\text{if} \quad h_{ij} = 1,\\
    (v_i, w'_j),\text{ and } (v'_i, w_j) &\quad\text{if} \quad h_{ij} = -1.
\end{cases}\]
Complete \(G_H\) by adding loops on all vertices \(v_i, v'_i\), \(i \in \{1, \dots, m\}\).
If \(H\) is a Hadamard matrix, we call the graph \(G_H\) the \emph{Hadamard graph}.
\end{definition}

For example, consider the equivalent Hadamard matrices \[H_1 = 
\begin{pmatrix}
    1 & 1\\ 1 & -\\
\end{pmatrix},
\qquad
H_2 = \begin{pmatrix}
    - & -\\ 1& -\\
\end{pmatrix},
\qquad
H_3 = \begin{pmatrix}
    1 & -\\ 1 & 1\\
\end{pmatrix}.
\]

The corresponding Hadamard graphs are visualized in Figure \ref{McKayGraphsEx}. It is easy  to see that these three graphs are all isomorphic, which turns out to be a general result.

\begin{figure}[h]
    \centering
    \includegraphics[scale = 0.5]{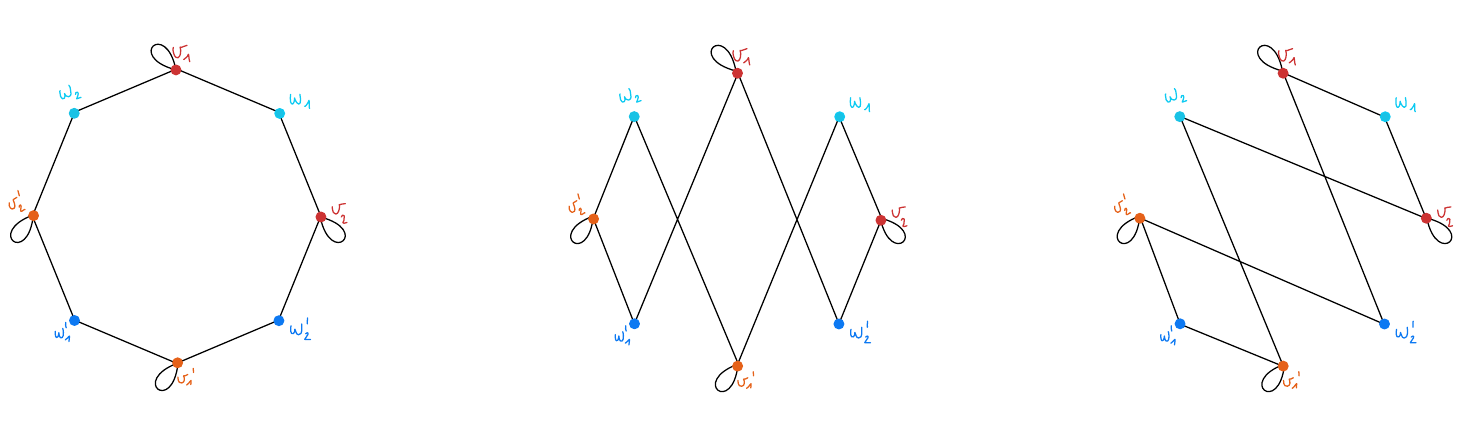}
    \caption{Hadamard graphs corresponding to \(H_1, H_2\) and \(H_3\), respectively.}
    \label{McKayGraphsEx}
\end{figure}

\begin{theorem}\cite{McKay1979}
    Let \(H_1\) and \(H_2\) be Hadamard matrices of order \(n\), and \(G_1 = G_{H_1}\), \(G_2 = G_{H_2}\) their Hadamard graphs. Then \(H_1\) and \(H_2\) are equivalent if and only if \(G_1\) and \(G_2\) are isomorphic. 
\end{theorem}

\section{New equivalences}\label{sec:newresults}

In Section \ref{graphhad} we have seen a method to construct graphs from Hadamard matrices. Next we investigate what the effect is of switching a closed quadruple or a Hall set of a Hadamard matrix on the corresponding graph. This way we are able to show new equivalences between the two switching methods for constructing Hadamard matrices introduced in \cite{Orrick2008} and the well-known GM-switching method for constructing cospectral graphs \cite{Godsil1982}. We do so by using McKay's construction of Hadamard graphs \cite{McKay1979} described in Section \ref{graphhad}. Afterwards we computationally explore if the switching methods for Hadamard matrices in \cite{Orrick2008} always produce inequivalent Hadamard matrices.

\subsection{Switching a closed quadruple and GM-switching}\label{sec:switchclosedGM}

In \cite{brent2011finding} the following is mentioned:

``This (switching a closed quadruple) is $H$-equivalent to flipping the signs of all but the leftmost block, which has a nicer interpretation in terms of switching edges in the corresponding bipartite graph (meaning McKay's Hadamard graph)''. 

Here $H$-equivalent stands for the usual equivalence of Hadamard matrices. In this section we show that indeed, this connection between switching a closed quadruple and a Hadamard graph can be formalized, and we use it to find a new equivalence between the methods of switching a closed quadruple and GM-switching.

\begin{theorem} \label{closedGM}
    Let \(H\) be a Hadamard matrix with a closed quadruple, and let \(G_H\) be its corresponding Hadamard graph. Then, there exists a switching partition of \(G_H\) such that switching a closed quadruple on \(H\) is equivalent to GM-switching on \(G_H\).
\end{theorem}

\begin{proof}
    Because every Hadamard matrix is equivalent to one of the form \eqref{closedquad} and equivalent matrices deliver isomorphic Hadamard graphs, we can consider \(H\) to be of the form \eqref{closedquad}. Let \(n\) be the order of \(H\). Consider the following partition on the vertices of \(G_H\):
\begin{align*}
D &= \{v_1, v_2, v_3, v_4, v'_1, v'_2, v'_3, v'_4\},\\
C_1 &= \{w_1,\dots,w_{\frac{n}{4}}, w'_1, \dots w'_{\frac{n}{4}}\},\\
C_2 &= \{v_5,\dots,v_n,v'_5,\dots,v'_n\},\\
C_3 &= \{w_{\frac{n}{4} + 1}, \dots, w_{\frac{n}{2}}\},\\ 
C_4 &= \{w'_{\frac{n}{4} + 1}, \dots, w'_{\frac{n}{2}}\},\\
C_5 &= \{w_{\frac{n}{2} + 1}, \dots, w_{\frac{3n}{4}}\},\\ 
C_6 &= \{w'_{\frac{n}{2} + 1}, \dots, w'_{\frac{3n}{4}}\},\\
C_7 &= \{w_{\frac{3n}{4} + 1}, \dots, w_n\},\\
C_8 &= \{w'_{\frac{3n}{4} + 1}, \dots, w'_n\}.
\end{align*} 

The Hadamard graph is bipartite with partition sets \(U = \{v_i,v'_i|\,i \in \{1,\dots,n\}\}\) and \(W = \{w_j,w'_j|\,j \in \{1,\dots,n\}\}\), and loops on all vertices in $U$. Thus there are no edges between \(C_i\) and \(C_j\) for every \(i,j \in \{1\} \cup \{3,\dots,8\}\) and every vertex in $D$ has no neighbours in $C_2$.
By construction every vertex in $U$ is adjacent to either $w_j$ or $w'_j$ for all $j \in \{1,\dots,n\}$ and every vertex in $W$ is adjacent to either $v_i$ or $v'_i$, $i \in \{1,\dots, n\}$.
So every vertex of \(C_i\), \(i \in \{3,\dots,8\}\) has \(n-4\) neighbours in \(C_2\), every vertex of \(C_2\) has \(\frac{n}{4}\) neighbours in \(C_1\) and very vertex in $D$ has $\frac{1}{2}|C_1|$ neighbours in $C_1$.
Notice that set $D$ corresponds to the rows of the closed quadruple. Set $C_1$ refers to the columns in the first field, $\{C_3,C_4\}$ to the second field, \(\{C_5,C_6\}\) to the third and \(\{C_7,C_8\}\) to the forth. 
Since all entries of a row of the closed quadruple are the same in one field, every vertex in $D$ is adjacent to either all or no vertices of $C_i$, $i \in \{3,\dots,8\}$.
We now prove that every vertex of \(C_2\) has \(\frac{n}{8}\) neighbours in \(C_i\) for \(i \in \{3,\dots,8\}\). 
The Hadamard matrix is 3-normalized on the first three rows, thus as remarked in the beginning of Section \ref{Hadamard}, for any other row the sum of all elements in a field is the same for all fields. Because all rows outside the closed quadruple also have to be orthogonal with the all one row of the closed quadruple, this sum has to be 0. So every row, other than the first four, has \(\frac{n}{8}\) positive and \(\frac{n}{8}\) negative entries in each field. For the graph \(G_H\) this means that indeed all vertices of \(C_2\) have \(\frac{n}{8}\) neighbours in \(C_i\) for \(i \in \{3,\dots,8\}\).

The given partition of the vertices of \(G_H\) thus satisfies all conditions of GM-switching and GM-switching on \(G_H\) delivers the same graph as first switching a closed quadruple of \(H\) and then constructing the Hadamard graph.
\end{proof}

\begin{example}
    Take for example the matrix \(H_{16.0}\) of order 16 (see Appendix \ref{H16}), which has a closed quadruple in its first four rows.

Here switching a closed quadruple would mean negating the entries of the first four rows in columns 1, 5, 9 and 13. In the corresponding graph, negating a positive entry on position \((i,j)\) translates into deleting edges \((v_i, w_j)\) and \((v'_i, w'_j)\), and adding edges \((v_i, w'_j)\) and \((v'_i, w_j)\). The transformation for a negative entry is analogous. Figure \ref{subgrapgsSwitchset} visualizes the edges that change in our example. 

\begin{figure}[htp!]
    \begin{subfigure}{0.49\textwidth}
    \centering
    \includegraphics[width = \textwidth]{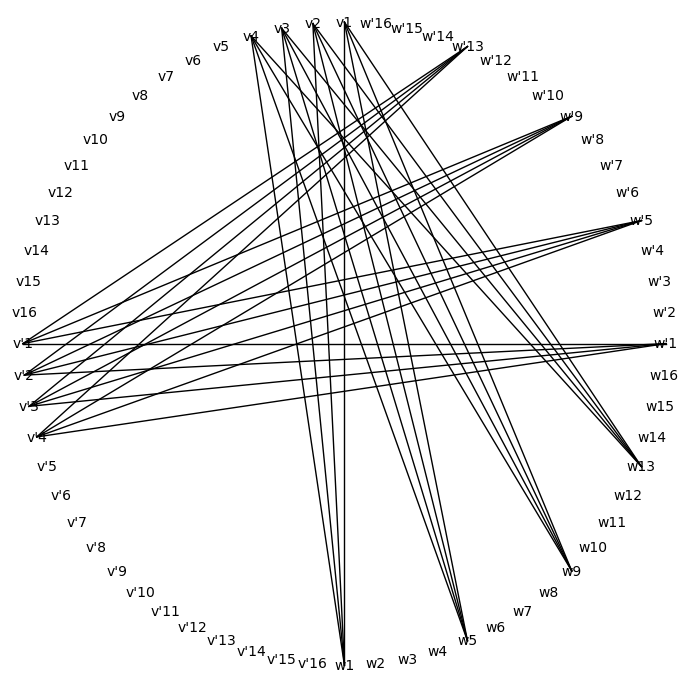}
    \end{subfigure}
    \begin{subfigure}{0.49\textwidth}
    \centering
    \includegraphics[width = \textwidth]{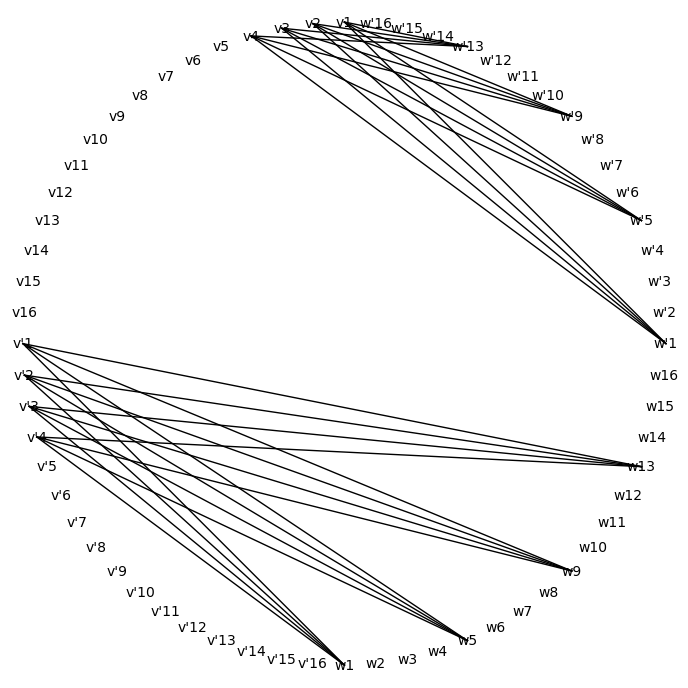}
    \end{subfigure}
    \caption{Subgraph corresponding to the switched field before and after switching a closed quadruple.}
    \label{subgrapgsSwitchset}
\end{figure}

\allowdisplaybreaks
Let \(D\) be the set of vertices $$\{v_1, v_2, v_3, v_4, v'_1, v'_2, v'_3, v'_4\}$$ and \(C_1\) the set 

$$\{w_1, w_5, w_9, w_{13},w'_1, w'_5, w'_9,w'_{13}\},$$

\noindent then this subgraph already satisfies the conditions for GM-switching (see Definition \ref{GMswitch}). All that remains is completing the partition of the vertices in a way that it still satisfies the conditions for GM-switching. We can define \(C_2\) as all the remaining \(v_i\) and \(v'_i\), because \(D\) will have no edges to \(C_2\) (because the graph is bipartite) and because for every \(v_i\) in \(C_2\), \(C_2\) also contains \(v'_i\), so every vertex \(w_j\) will be adjacent to half of the vertices of \(C_2\).
We now divide the remaining \(w_j\) and \(w'_j\) such that every vertex in \(D\) is adjacent to either 0 or all vertices in each part. Because \(D\) corresponds to the rows of the closed quadruple, in the Hadamard matrix this translates to all columns in the closed quadruple being the same. So the rest of the partition is defined by the fields of the Hadamard matrix. To see this, take for example field 2, which consists of columns 2, 6, 10 and 14. Rows 1 and 3 have all positive entries in this field, so \(v_1\) and \(v_3\) will be adjacent to all \(w_j\), and \(v'_1\) and \(v'_3\) will be adjacent to all \(w'_j\), \(j \in \{2,6,10,14\}\). The opposite is true for row 2 and 4. Thus if we consider the sets \(C_3 = \{w_2,w_6,w_{10},w_{14}\}\) and \(C_4 = \{w'_2,w'_6,w'_{10},w'_{14}\}\), any vertex of \(D\) is either adjacent to all or no vertices of \(C_3\). The same holds for \(C_4\). These sets of vertices also satisfy the condition that every vertex of \(C_i\) needs to have an equal amount of neighbours in \(C_j\) for all \(i,j\). 
To conclude, if we define the graph switching partition as follows:
\begin{align*}    
D &= \{v_1, v_2, v_3, v_4, v'_1, v'_2, v'_3, v'_4\},\\
C_1 &= \{w_1, w_5, w_9, w_{13}, w'_1, w'_5, w'_9, w'_{13}\},\\
C_2 &= \{v_5,v_6,\dots,v_{16},v'_5,v'_6,\dots,v'_{16}\},\\
C_3 &= \{w_2,w_6,w_{10},w_{14}\},\\
C_4 &= \{w'_2,w'_6,w'_{10},w'_{14}\},\\
C_5 &= \{w_3,w_7,w_{11},w_{15}\},\\
C_6 &= \{w'_3,w'_7,w'_{11},w'_{15}\},\\
C_7 &= \{w_4,w_8,w_{12},w_{16}\},\\
C_8 &= \{w'_4,w'_8,w'_{12},w'_{16}\},\
\end{align*}
switching a closed quadruple in the Hadamard matrix and then creating the Hadamard graph delivers the exact same result as first creating the Hadamard graph and then using GM-switching.
\end{example}

For deriving a converse statement of Theorem \ref{closedGM}, we need to find out what conditions a GM-switching partition has to satisfy in the Hadamard graph in order to correspond to a closed quadruple. The next result establishes this.

\begin{theorem} \label{thm:converseGMclosed}
    Let $G_H$ be the Hadamard graph of a Hadamard matrix $H$ of order $n$. GM-switching on $G_H$ is equivalent to switching a closed quadruple on $H$ if the switching partition of the vertex set of $G_H$ satisfies the following conditions:

    \begin{itemize}
        \item The GM-partition of the vertex set of $G_H$ has sets \linebreak $\{D,C_1,C_2,\dots,C_m\}$, $m \geq 9$.
        \item $|D| = 8, |C_1| = n/2, |C_3| = |C_4| = \dots = |C_8| = n/4$.
        \item $D$ consist of eight vertices with loops and for each vertex in $D$, there is another vertex in $D$ that has no common neighbours with it.
        \item No vertices of $C_1, C_2, C_3, \dots, C_8$ have loops.
        \item Half of the vertices of $C_1$ are adjacent to the same four vertices of $D$. For each vertex in $C_1$ there is another vertex in $C_1$ that has no common neighbours with it.
        \item Four vertices in $D$ are adjacent to all vertices in $C_3, C_5$ and $C_7$, the other four are adjacent to all vertices in $C_4, C_6$ and $C_8$. 
        \end{itemize}
\end{theorem}

\begin{proof}
We know that the switching partition of the vertex set of the Hadamard graph $G_H$ should resemble the partition defined in the proof of Theorem \ref{closedGM}. That partition had nine sets, named \(\{D,C_1,C_2,\dots,C_8\}\). The sets that corresponded to the closed quadruple were \linebreak $D, C_1,C_3,C_4,\dots,C_8$, so we need these sets again. However, set $C_2$ could be split in to multiple smaller sets and still satisfy the conditions of GM-switching. Therefore we need at least nine sets in our partition of the vertex set of $G_H$.

The Hadamard graph has $4n$ vertices, two for each row and two for each column. The vertices that correspond to row $i$ are labeled $v_i$ and $v'_i$. The vertices that correspond to column $j$ are labeled $w_j$ and $w'_j$. Since all vertices $v_i$ and $v'_i$, $i \in \{1,\dots,n\}$, have loops and all $w_j$ and $w'_j$, $j \in \{1,\dots,n\}$ do not, we can always know if a vertex of $G_H$ was a $v, v'$ or $w, w'$. 
We need our set $D$ to correspond to four different rows of the Hadamard matrix, so we need eight vertices that have loops. For every $v$ in $D$, the corresponding $v'$ must also be in $D$. 
Notice that every vertex $v$ can only be adjacent to $n/2$ other vertices, namely those labeled $w_j$ or $w'_j$. Vertex $v$ is adjacent to $n/4$ of these $n/2$ vertices. By the construction of the Hadamard graph and the fact that a Hadamard matrix can never have the same row twice, the corresponding $v'$ is the only vertex that is adjacent to the other $n/4$ of these vertices.
Thus our $D$ needs four different vertices with loops and for each of them the vertex witch a loop that has no common neighbours with it.

Set $C_1$ must resembles the columns of the field on which we switch, so we need $n/2$ vertices that have no loops. In this field, all rows have the same entry sequence (or the negation of it), so half of the vertices in $C_1$ must be adjacent to the same four vertices in $D$. Again, for every vertex $w$ in $C_1$, we need the corresponding $w'$ to also be in $C_1$. For the same reasoning as in $D$, these vertices $w'$ are the vertices with no loops that have no common neighbours with one of the other vertices in $C_1$.

Sets $C_3, C_4, \dots, C_8$ should correspond to the entries of the closed quadruple in the other fields, separated such that every vertex $v$ in $D$ is adjacent to all or none of the vertices of $C_i, i \in \{3,\dots,8\}$. For every $C_i$ that $v$ is adjacent to, there must be a $C_j, j \in \{3,\dots,8\}$ in which $v$ has no neighbours, namely the $C_j$ that corresponds to the same field as $C_i$. Since the corresponding $v'$ is also in $D$, we can label all $C_i$ such that four vertices in $D$ are adjacent to all vertices in $C_3, C_5$ and $C_7$, the other four are adjacent to all vertices in $C_4, C_6$ and $C_8$.

All other sets in the switching partition can now only contain vertices with loops. These vertices correspond to all the other rows of the Hadamard matrix, and are therefore not important for the representation of the closed quadruple. Thus there are no further restrictions needed on these vertex sets.
\end{proof}

\subsection{Switching a Hall set and GM-switching}\label{sec:GMHall}

For Hadamard matrices with a Hall set, a similar switching partition as in Section \ref{sec:switchclosedGM} can be defined. 

\begin{theorem}\label{HallGM}
    Let \(H\) be a Hadamard matrix with a Hall set, and let \(G_H\) be its corresponding Hadamard graph. Then, there exists a switching partition of \(G_H\) such that switching a Hall set on \(H\) is equivalent to GM-switching on \(G_H\).
\end{theorem}

 \begin{proof}
     We can assume \(H\) to be of the form \eqref{Hallsetmatrix}. Let \(n\) be the order of \(H\). Consider the following partition of \(G_H\):
\begin{align*}
D &= \{v_1, \dots, v_4, v'_1, \dots, v'_4, w_1, \dots, w_4, w'_1, \dots, w'_4\},\\
C_1 &= \{v_5, \dots, v_{\frac{n}{4}+3}, v'_5, \dots, v'_{\frac{n}{4}+3}\},\\
C_2 &= \{w_5, \dots, w_{\frac{n}{4}+3}, w'_5, \dots, w'_{\frac{n}{4}+3}\},\\
C_3 &= \{v_{\frac{n}{4}+4}, \dots, v_{\frac{n}{2}+2}\},\\
C_4 &= \{v'_{\frac{n}{4}+4}, \dots, v'_{\frac{n}{2}+2}\},\\
C_5 &= \{w_{\frac{n}{4}+4}, \dots, w_{\frac{n}{2}+2}\},\\
C_6 &= \{w'_{\frac{n}{4}+4}, \dots, w'_{\frac{n}{2}+2}\},\\
C_7 &= \{v_{\frac{n}{2}+3}, \dots, v_{\frac{3n}{4}+1}\},\\
C_8 &= \{v'_{\frac{n}{2}+3}, \dots, v'_{\frac{3n}{4}+1}\},\\
C_9 &= \{w_{\frac{n}{2}+3}, \dots, w_{\frac{3n}{4}+1}\},\\
C_{10} &= \{w'_{\frac{n}{2}+3}, \dots, w'_{\frac{3n}{4}+1}\},\\
C_{11} &= \{v_{\frac{3n}{4}+2}, \dots, v_{n}\},\\
C_{12} &= \{v'_{\frac{3n}{4}+2}, \dots, v'_{n}\},\\
C_{13} &= \{w_{\frac{3n}{4}+2}, \dots, w_{n}\},\\
C_{14} &= \{w'_{\frac{3n}{4}+2}, \dots, w'_{n}\}.\end{align*}
Here \(D\) corresponds to the Hall set and Hall columns, \(C_1\) to the other rows and \(C_2\) to the other columns on which we want to switch. All other sets are defined according to the fields of the Hall set and Hall columns. 
All vertices \(w_i, w'_i\) of \(D\) are adjacent to half of the vertices of \(C_1\) and no vertices of \(C_2\). The converse is true for the vertices \(v_i, v'_i\) of \(D\). 
Because all submatrices \(A_{ij}\) in \eqref{Hallsetmatrix} have constant column and row sums, all vertices of \(C_i\) will have an equal amount of neighbours in \(C_j\) for all \(i,j \in \{3, \dots, 14\}\). All other conditions of GM-switching are satisfied for similar reasons as in the proof of Theorem \ref{closedGM}. 

Given the partition above, applying GM-switching on the Hadamard graph is equivalent to switching the Hall set of the Hadamard matrix and then creating the Hadamard graph.
 \end{proof}

 Analogously as we did in Theorem \ref{thm:converseGMclosed}, we can find conditions for the GM-partition of the Hadamard matrix in order for it to correspond to a Hall set.

\begin{theorem}\label{thm:converseGMHallset}
    Let $G_H$ be the Hadamard graph of a Hadamard matrix $H$ of order $n$. GM-switching on $G_H$ is equivalent to switching a Hall set on $H$ if the switching partition of the vertex set of $G_H$ satisfies the following conditions:
    \begin{itemize}
        \item The partition has 15 sets $\{D,C_1,C_2,\dots,C_{14}\}$.
        \item $|D| = 16, |C_1| = |C_2| = n/2-2, |C_3| = |C_4| = \dots = |C_{14}| = n/4-1$.
        \item $D$ consist of eight vertices with loops and eight without. For each vertex in $D$ with (resp. without) loops, there is another vertex in $D$ with (resp. without) loops that has no common neighbours with it.
        \item All vertices in $C_1$ have loops. For every vertex in $C_1$ there is another vertex in $C_1$ that has no common neighbours with it. Half of the vertices of $C_1$ are adjacent to the same four vertices of $D$.
        \item No vertices in $C_2$ have loops. For every vertex in $C_2$ there is another vertex in $C_2$ that has no common neighbours with it. Half of the vertices of $C_2$ are adjacent to the same four vertices of $D$.
        \item All vertices in $C_i, i \in \{2,4,7,8,11,12\}$ have loops, while no vertices in $C_j, j \in \{5,6,8,10,13,14\}$ have loops.
        \item Four vertices in $D$ are adjacent to all vertices in $C_3, C_7$ and $C_{11}$, four different vertices are adjacent to all vertices in $C_4, C_8$ and $C_{12}$. Four other vertices are adjacent to all vertices in $C_5, C_9$ and $C_{13}$. The remaining four vertices in $D$ are adjacent to all vertices in $C_6, C_{10}$ and $C_{14}$.
        \end{itemize}
\end{theorem}

\begin{proof}
    The proof of this theorem is similar to that of Theorem \ref{thm:converseGMclosed}. We know a switching partition of the Hadamard graph corresponding to a Hall set of the Hadamard matrix resembles the partition $\{D,C_1,\dots,C_{14}\}$ shown in the proof of Theorem \ref{HallGM}. In this partition all sets are important to define the Hall set, so we need exactly 15 sets.
    All vertices that are labeled $v_i, v'_i, i \in \{1,\dots,n\}$ have loops and all vertices labeled $w_j, w'_j, j \in \{1,\dots,n\}$ have no loops.

    The set $D$ corresponds to the Hall set rows and columns and therefore contains eight vertices with loops for the rows and eight vertices without loops for the columns. For every $v$ in $D$, the $v'$ corresponding to the same row must also be in $D$. This is the only vertex with a loop that has no common neighbours with $v$. Similarly, for every $w$ in $D$, the $w'$ corresponding to the same column as $w$ must also be in $D$. This is the only vertex without loops that has no common neighbours with $w$.

    We want $C_1$ and $C_2$ to correspond respectively to field of the Hall set rows and columns that we want to switch on. Therefore $C_1$ must contain $n/2-2$ vertices with loops. For every vertex $v$ in $C_1$, the corresponding $v'$ must also be in $C_1$. Set $C_2$ needs $n/2-2$ loopless vertices and for every $w$ in $C_2$, the corresponding $w'$ must also be in $C_2$. In these fields all Hall set rows (resp. columns) have the same entry sequence (or the negation of it), so half of the vertices in $C_1$ (resp. $C_2$) are adjacent to the same four vertices in $D$.

    All other sets should correspond to the different fields in either the Hall set rows or columns. We can label the sets such that $C_i, i \in \{3,4,7,8,11,12\}$ correspond to the to the different fields in the rows and $C_j, j \in \{5,6,9,10,13,14\}$ to the different fields in the columns. 
    Analogously to the proof of Theorem \ref{thm:converseGMclosed}, we get that four vertices in $D$ are adjacent to all vertices in $C_3, C_7$ and $C_{11}$, four different vertices are adjacent to all vertices in $C_4, C_8$ and $C_{12}$. Four other vertices are adjacent to all vertices in $C_5, C_9$ and $C_{13}$. The remaining four vertices in $D$ are adjacent to all vertices in $C_6, C_{10}$ and $C_{14}$.
\end{proof}

\subsection{Non-equivalence of Hadamard matrices constructed using Orrick's switchings}\label{noneqswitch}

\subsubsection{Switching a closed quadruple}\label{noneqswitchclosed}

In \cite{Orrick2008} Orrick wrote:

``It appears that when \(n > 8\), switching [a closed quadruple] always produces a Hadamard matrix that is inequivalent to the original Hadamard matrix.''
   
However, no proof nor evidence of why this is the case is provided in \cite{Orrick2008}. We confirmed this statement computationally for all Hadamard matrices of order 16 and 24. There are five equivalence classes of Hadamard matrices of order 16 (see Appendix \ref{H16}) and 60 equivalence classes of Hadamard matrices of order 24. For order 16, the Smith normal form of the Hadamard matrix was sufficient for proving non-equivalence. However, for order 24 the Smith normal form of the Hadamard matrices is always the same, thus a more sensitive method is needed. Of the 60 equivalence classes of Hadamard matrices of order 24, 2 have no closed quadruple. Of the remaining 58 classes, 44 had different 4-profiles, leaving only 14 to check. To prove non-equivalence of the remaining matrices after switching, we computed the corresponding Hadamard graphs and checked if they were isomorphic, which they were not. This method was introduced in Section \ref{graphhad}.

\subsubsection{Switching a Hall set}\label{noneqswitchHall}

For switching a Hall set it is not true that switching always delivers inequivalent Hadamard matrices. As Orrick \cite{Orrick2008} stated:

``Examples are known where switching a Hall set in a Hadamard matrix \(H\) produces a matrix equivalent to \(H\). In general, however, one obtains an inequivalent matrix.''

One of these examples is the Hadamard matrix \(H_{36}\) of order 36 in the Appendix \ref{H36}. Rows and columns 1, 4, 19, and 22 form a Hall set. After switching, the corresponding Hadamard graphs are isomorphic, and thus the Hadamard matrices are equivalent.

\section{Concluding remarks}\label{sec:concremarks}


The main focus of the current study has been on investigating a relationship between switching methods for graphs and Hadamard matrices. In particular, we focused on the method to construct cospectral graphs by Godsil and McKay \cite{Godsil1982} (GM-switching) and two methods for constructing Hadamard matrices by Orrick \cite{Orrick2008} (switching a closed quadruple, switching a Hall set), and showed an equivalence between them. This idea was mentioned in 2019 by the first author to Orrick \cite{Aidaprivatecommunication}, who was unaware of GM-switching and therefore of a potential connection between the methods. 

Some variations and extensions of GM-switching have been recently proposed in the literature  \cite{Abiad2012,wang2019cospectral, qiu2020theorem,ABS2024switching}. It would be interesting to investigate if similar equivalences of such switching operations can be shown with other methods to construct Hadamard matrices. 


It would also be interesting to find a theoretical proof of \linebreak non-equivalence for the Hadamard matrices obtained by using the switching a closed quadruple method, addressing  a problem proposed by Orrick \cite{Orrick2008}. 



\subsection*{Acknowledgements} 
Aida Abiad is supported by NWO (Dutch Research Council) through the grant VI.Vidi.213.085. The authors would like to thank Ronan Egan for bringing Orrick's work to their attention, and Robin Simoens for a careful reading of this manuscript. 
 
\bibliographystyle{abbrv}
\bibliography{sources}

\newpage
\section{Appendix}

We include all Hadamard matrices of order 1 to 20 up to equivalence, along with some equivalent Hadamard matrices of order 36, used as examples in Section \ref{noneqswitchHall}. For orders \(\leq 12\) there is only one equivalence class per order. There are 5 equivalence classes of Hadamard matrices of order 16 and 3 of order 20. For the 60 inequivalent Hadamard matrices of order 24 and Hadamard matrices of larger sizes, see Sloane's database \cite{Sloane}.

\small

\subsection{Order 1 to 12}
\begin{equation*}
    \begin{pmatrix}
        1
    \end{pmatrix}
    \qquad
    \begin{pmatrix}
        1&1\\
        1&-
    \end{pmatrix}
    \qquad
    \begin{pmatrix}
        1&1&1&1\\
        1&1&-&-\\
        1&-&1&-\\
        1&-&-&1
    \end{pmatrix}
    \qquad
    \begin{pmatrix}
        1&1&1&1&1&1&1&1\\
        1&1&-&-&1&1&-&-\\
        1&-&1&-&1&-&1&-\\
        1&-&-&1&1&-&-&1\\
        1&1&1&1&-&-&-&-\\
        1&1&-&-&-&-&1&1\\
        1&-&1&-&-&1&-&1\\
        1&-&-&1&-&1&1&-\\
    \end{pmatrix}
    \end{equation*}

    \begin{equation*}
    \begin{pmatrix}
        1&-&-&-&1&1&1&1&1&1&-&-\\
    -&1&-&-&1&1&-&-&-&-&-&-\\
    -&-&1&-&1&1&-&-&1&1&1&1\\
    -&-&-&1&1&1&1&1&-&-&1&1\\
    1&1&1&1&1&1&-&1&1&-&-&1\\
    1&1&1&1&1&1&1&-&-&1&1&-\\
    -&1&1&-&-&1&1&1&1&-&1&-\\
    -&1&1&-&1&-&1&1&-&1&-&1\\
    -&1&-&1&1&-&-&1&1&1&1&-\\
    -&1&-&1&-&1&1&-&1&1&-&1\\
    1&1&-&-&-&1&-&1&-&1&1&1\\
    1&1&-&-&1&-&1&-&1&-&1&1\\
    \end{pmatrix}
\end{equation*}

\subsection{Order 16}\label{H16}
\begin{align*}
H_{16.0} = &
    \begin{pmatrix} 
    1 & 1 & 1 & 1 & 1 & 1 & 1 & 1 & 1 & 1 & 1 & 1 & 1 & 1 & 1 & 1 \\
    1 & - & 1 & - & 1 & - & 1 & - & 1 & - & 1 & - & 1 & - & 1 & - \\
    1 & 1 & - & - & 1 & 1 & - & - & 1 & 1 & - & - & 1 & 1 & - & - \\
    1 & - & - & 1 & 1 & - & - & 1 & 1 & - & - & 1 & 1 & - & - & 1 \\
    1 & 1 & 1 & 1 & - & - & - & - & 1 & 1 & 1 & 1 & - & - & - & - \\
    1 & - & 1 & - & - & 1 & - & 1 & 1 & - & 1 & - & - & 1 & - & 1 \\
    1 & 1 & - & - & - & - & 1 & 1 & 1 & 1 & - & - & - & - & 1 & 1 \\
    1 & - & - & 1 & - & 1 & 1 & - & 1 & - & - & 1 & - & 1 & 1 & - \\
    1 & 1 & 1 & 1 & 1 & 1 & 1 & 1 & - & - & - & - & - & - & - & - \\
    1 & - & 1 & - & 1 & - & 1 & - & - & 1 & - & 1 & - & 1 & - & 1 \\
    1 & 1 & - & - & 1 & 1 & - & - & - & - & 1 & 1 & - & - & 1 & 1 \\
    1 & - & - & 1 & 1 & - & - & 1 & - & 1 & 1 & - & - & 1 & 1 & - \\
    1 & 1 & 1 & 1 & - & - & - & - & - & - & - & - & 1 & 1 & 1 & 1 \\
    1 & - & 1 & - & - & 1 & - & 1 & - & 1 & - & 1 & 1 & - & 1 & - \\
    1 & 1 & - & - & - & - & 1 & 1 & - & - & 1 & 1 & 1 & 1 & - & - \\
    1 & - & - & 1 & - & 1 & 1 & - & - & 1 & 1 & - & 1 & - & - & 1
    \end{pmatrix}\\
\\
H_{16.1} = &
    \begin{pmatrix}
        1 & 1 & 1 & 1 & 1 & 1 & 1 & 1 & 1 & 1 & 1 & 1 & 1 & 1 & 1 & 1 \\
        1 & - & 1 & - & 1 & - & 1 & - & 1 & - & 1 & - & 1 & - & 1 & - \\
        1 & 1 & - & - & 1 & 1 & - & - & 1 & 1 & - & - & 1 & 1 & - & - \\
        1 & - & - & 1 & 1 & - & - & 1 & 1 & - & - & 1 & 1 & - & - & 1 \\
        1 & 1 & 1 & 1 & - & - & - & - & 1 & 1 & 1 & 1 & - & - & - & - \\
        1 & - & 1 & - & - & 1 & - & 1 & 1 & - & 1 & - & - & 1 & - & 1 \\
        1 & 1 & - & - & - & - & 1 & 1 & 1 & 1 & - & - & - & - & 1 & 1 \\
        1 & - & - & 1 & - & 1 & 1 & - & 1 & - & - & 1 & - & 1 & 1 & - \\
        1 & 1 & 1 & 1 & 1 & 1 & 1 & 1 & - & - & - & - & - & - & - & - \\
        1 & - & 1 & - & 1 & - & - & 1 & - & 1 & - & 1 & - & 1 & 1 & - \\
        1 & 1 & - & - & 1 & 1 & - & - & - & - & 1 & 1 & - & - & 1 & 1 \\
        1 & - & - & 1 & 1 & - & 1 & - & - & 1 & 1 & - & - & 1 & - & 1 \\
        1 & 1 & 1 & 1 & - & - & - & - & - & - & - & - & 1 & 1 & 1 & 1 \\
        1 & - & 1 & - & - & 1 & 1 & - & - & 1 & - & 1 & 1 & - & - & 1 \\
        1 & 1 & - & - & - & - & 1 & 1 & - & - & 1 & 1 & 1 & 1 & - & - \\
        1 & - & - & 1 & - & 1 & - & 1 & - & 1 & 1 & - & 1 & - & 1 & -\\
    \end{pmatrix}\\
\end{align*}
\begin{align*}
H_{16.2} = &
    \begin{pmatrix}
        1 & 1 & 1 & 1 & 1 & 1 & 1 & 1 & 1 & 1 & 1 & 1 & 1 & 1 & 1 & 1 \\
        1 & - & 1 & - & 1 & - & 1 & - & 1 & - & 1 & - & 1 & - & 1 & - \\
        1 & 1 & - & - & 1 & 1 & - & - & 1 & 1 & - & - & 1 & 1 & - & - \\
        1 & - & - & 1 & 1 & - & - & 1 & 1 & - & - & 1 & 1 & - & - & 1 \\
        1 & 1 & 1 & 1 & - & - & - & - & 1 & 1 & 1 & 1 & - & - & - & - \\
        1 & - & 1 & - & - & 1 & - & 1 & 1 & - & 1 & - & - & 1 & - & 1 \\
        1 & 1 & - & - & - & - & 1 & 1 & 1 & 1 & - & - & - & - & 1 & 1 \\
        1 & - & - & 1 & - & 1 & 1 & - & 1 & - & - & 1 & - & 1 & 1 & - \\
        1 & 1 & 1 & 1 & 1 & 1 & 1 & 1 & - & - & - & - & - & - & - & - \\
        1 & 1 & 1 & 1 & - & - & - & - & - & - & - & - & 1 & 1 & 1 & 1 \\
        1 & 1 & - & - & 1 & - & 1 & - & - & - & 1 & 1 & - & 1 & - & 1 \\
        1 & 1 & - & - & - & 1 & - & 1 & - & - & 1 & 1 & 1 & - & 1 & - \\
        1 & - & 1 & - & 1 & - & - & 1 & - & 1 & - & 1 & - & 1 & 1 & - \\
        1 & - & 1 & - & - & 1 & 1 & - & - & 1 & - & 1 & 1 & - & - & 1 \\
        1 & - & - & 1 & 1 & 1 & - & - & - & 1 & 1 & - & - & - & 1 & 1 \\
        1 & - & - & 1 & - & - & 1 & 1 & - & 1 & 1 & - & 1 & 1 & - & -
    \end{pmatrix}\\
\\
H_{16.3}=&
    \begin{pmatrix}
        1 & 1 & 1 & 1 & 1 & 1 & 1 & 1 & 1 & 1 & 1 & 1 & 1 & 1 & 1 & 1 \\
        1 & - & 1 & - & 1 & - & 1 & - & 1 & - & 1 & - & 1 & - & 1 & - \\
        1 & 1 & - & - & 1 & 1 & - & - & 1 & 1 & - & - & 1 & 1 & - & - \\
        1 & - & - & 1 & 1 & - & - & 1 & 1 & - & - & 1 & 1 & - & - & 1 \\
        1 & 1 & 1 & 1 & - & - & - & - & 1 & 1 & 1 & 1 & - & - & - & - \\
        1 & - & 1 & - & - & 1 & - & 1 & 1 & - & 1 & - & - & 1 & - & 1 \\
        1 & 1 & - & - & - & - & 1 & 1 & 1 & 1 & - & - & - & - & 1 & 1 \\
        1 & - & - & 1 & - & 1 & 1 & - & 1 & - & - & 1 & - & 1 & 1 & - \\
        1 & 1 & 1 & 1 & 1 & 1 & 1 & 1 & - & - & - & - & - & - & - & - \\
        1 & 1 & 1 & - & 1 & - & - & - & - & - & - & 1 & - & 1 & 1 & 1 \\
        1 & 1 & - & 1 & - & - & - & 1 & - & - & 1 & - & 1 & 1 & 1 & - \\
        1 & 1 & - & - & - & 1 & 1 & - & - & - & 1 & 1 & 1 & - & - & 1 \\
        1 & - & 1 & 1 & - & 1 & - & - & - & 1 & - & - & 1 & - & 1 & 1 \\
        1 & - & 1 & - & - & - & 1 & 1 & - & 1 & - & 1 & 1 & 1 & - & - \\
        1 & - & - & 1 & 1 & - & 1 & - & - & 1 & 1 & - & - & 1 & - & 1 \\
        1 & - & - & - & 1 & 1 & - & 1 & - & 1 & 1 & 1 & - & - & 1 & -
    \end{pmatrix}\\
\\
H_{16.4}=&
    \begin{pmatrix}
        1 & 1 & 1 & 1 & 1 & 1 & 1 & 1 & 1 & 1 & 1 & 1 & 1 & 1 & 1 & 1 \\
        1 & - & 1 & - & 1 & - & 1 & - & 1 & 1 & 1 & 1 & - & - & - & - \\
        1 & 1 & - & - & 1 & 1 & - & - & 1 & 1 & - & - & 1 & 1 & - & - \\
        1 & - & - & 1 & 1 & - & - & 1 & 1 & - & 1 & - & 1 & - & 1 & - \\
        1 & 1 & 1 & 1 & - & - & - & - & 1 & 1 & - & - & - & - & 1 & 1 \\
        1 & - & 1 & - & - & 1 & - & 1 & 1 & - & - & 1 & 1 & - & - & 1 \\
        1 & 1 & - & - & - & - & 1 & 1 & 1 & - & - & 1 & - & 1 & 1 & - \\
        1 & - & - & 1 & - & 1 & 1 & - & 1 & - & 1 & - & - & 1 & - & 1 \\
        1 & 1 & 1 & 1 & 1 & 1 & 1 & 1 & - & - & - & - & - & - & - & - \\
        1 & - & 1 & - & 1 & - & 1 & - & - & - & - & - & 1 & 1 & 1 & 1 \\
        1 & 1 & - & - & 1 & 1 & - & - & - & - & 1 & 1 & - & - & 1 & 1 \\
        1 & - & - & 1 & 1 & - & - & 1 & - & 1 & - & 1 & - & 1 & - & 1 \\
        1 & 1 & 1 & 1 & - & - & - & - & - & - & 1 & 1 & 1 & 1 & - & - \\
        1 & - & 1 & - & - & 1 & - & 1 & - & 1 & 1 & - & - & 1 & 1 & - \\
        1 & 1 & - & - & - & - & 1 & 1 & - & 1 & 1 & - & 1 & - & - & 1 \\
        1 & - & - & 1 & - & 1 & 1 & - & - & 1 & - & 1 & 1 & - & 1 & -
    \end{pmatrix}
\end{align*}

\subsection{Order 20} \label{H20}
\begin{align*}
H_{20.0} = &
    \begin{pmatrix}
        1 & - & - & - & - & - & - & - & - & - & - & - & - & - & - & - & - & - & - & - \\
        1 & 1 & - & 1 & 1 & - & - & - & - & 1 & - & 1 & - & 1 & 1 & 1 & 1 & - & - & 1 \\
        1 & 1 & 1 & - & 1 & 1 & - & - & - & - & 1 & - & 1 & - & 1 & 1 & 1 & 1 & - & - \\
        1 & - & 1 & 1 & - & 1 & 1 & - & - & - & - & 1 & - & 1 & - & 1 & 1 & 1 & 1 & - \\
        1 & - & - & 1 & 1 & - & 1 & 1 & - & - & - & - & 1 & - & 1 & - & 1 & 1 & 1 & 1 \\
        1 & 1 & - & - & 1 & 1 & - & 1 & 1 & - & - & - & - & 1 & - & 1 & - & 1 & 1 & 1 \\
        1 & 1 & 1 & - & - & 1 & 1 & - & 1 & 1 & - & - & - & - & 1 & - & 1 & - & 1 & 1 \\
        1 & 1 & 1 & 1 & - & - & 1 & 1 & - & 1 & 1 & - & - & - & - & 1 & - & 1 & - & 1 \\
        1 & 1 & 1 & 1 & 1 & - & - & 1 & 1 & - & 1 & 1 & - & - & - & - & 1 & - & 1 & - \\
        1 & - & 1 & 1 & 1 & 1 & - & - & 1 & 1 & - & 1 & 1 & - & - & - & - & 1 & - & 1 \\
        1 & 1 & - & 1 & 1 & 1 & 1 & - & - & 1 & 1 & - & 1 & 1 & - & - & - & - & 1 & - \\
        1 & - & 1 & - & 1 & 1 & 1 & 1 & - & - & 1 & 1 & - & 1 & 1 & - & - & - & - & 1 \\
        1 & 1 & - & 1 & - & 1 & 1 & 1 & 1 & - & - & 1 & 1 & - & 1 & 1 & - & - & - & - \\
        1 & - & 1 & - & 1 & - & 1 & 1 & 1 & 1 & - & - & 1 & 1 & - & 1 & 1 & - & - & - \\
        1 & - & - & 1 & - & 1 & - & 1 & 1 & 1 & 1 & - & - & 1 & 1 & - & 1 & 1 & - & - \\
        1 & - & - & - & 1 & - & 1 & - & 1 & 1 & 1 & 1 & - & - & 1 & 1 & - & 1 & 1 & - \\
        1 & - & - & - & - & 1 & - & 1 & - & 1 & 1 & 1 & 1 & - & - & 1 & 1 & - & 1 & 1 \\
        1 & 1 & - & - & - & - & 1 & - & 1 & - & 1 & 1 & 1 & 1 & - & - & 1 & 1 & - & 1 \\
        1 & 1 & 1 & - & - & - & - & 1 & - & 1 & - & 1 & 1 & 1 & 1 & - & - & 1 & 1 & - \\
        1 & - & 1 & 1 & - & - & - & - & 1 & - & 1 & - & 1 & 1 & 1 & 1 & - & - & 1 & 1
    \end{pmatrix}\\
\\
    H_{20.1} = &
    \begin{pmatrix}
        1 & - & - & - & - & 1 & - & - & - & - & 1 & 1 & - & - & 1 & 1 & - & 1 & 1 & - \\
- & 1 & - & - & - & - & 1 & - & - & - & 1 & 1 & 1 & - & - & - & 1 & - & 1 & 1 \\
- & - & 1 & - & - & - & - & 1 & - & - & - & 1 & 1 & 1 & - & 1 & - & 1 & - & 1 \\
- & - & - & 1 & - & - & - & - & 1 & - & - & - & 1 & 1 & 1 & 1 & 1 & - & 1 & - \\
- & - & - & - & 1 & - & - & - & - & 1 & 1 & - & - & 1 & 1 & - & 1 & 1 & - & 1 \\
- & 1 & 1 & 1 & 1 & 1 & - & - & - & - & - & 1 & - & - & 1 & 1 & 1 & - & - & 1 \\
1 & - & 1 & 1 & 1 & - & 1 & - & - & - & 1 & - & 1 & - & - & 1 & 1 & 1 & - & - \\
1 & 1 & - & 1 & 1 & - & - & 1 & - & - & - & 1 & - & 1 & - & - & 1 & 1 & 1 & - \\
1 & 1 & 1 & - & 1 & - & - & - & 1 & - & - & - & 1 & - & 1 & - & - & 1 & 1 & 1 \\
1 & 1 & 1 & 1 & - & - & - & - & - & 1 & 1 & - & - & 1 & - & 1 & - & - & 1 & 1 \\
- & - & 1 & 1 & - & 1 & - & 1 & 1 & - & 1 & - & - & - & - & - & 1 & 1 & 1 & 1 \\
- & - & - & 1 & 1 & - & 1 & - & 1 & 1 & - & 1 & - & - & - & 1 & - & 1 & 1 & 1 \\
1 & - & - & - & 1 & 1 & - & 1 & - & 1 & - & - & 1 & - & - & 1 & 1 & - & 1 & 1 \\
1 & 1 & - & - & - & 1 & 1 & - & 1 & - & - & - & - & 1 & - & 1 & 1 & 1 & - & 1 \\
- & 1 & 1 & - & - & - & 1 & 1 & - & 1 & - & - & - & - & 1 & 1 & 1 & 1 & 1 & - \\
- & 1 & - & - & 1 & - & - & 1 & 1 & - & 1 & - & - & - & - & 1 & - & - & - & - \\
1 & - & 1 & - & - & - & - & - & 1 & 1 & - & 1 & - & - & - & - & 1 & - & - & - \\
- & 1 & - & 1 & - & 1 & - & - & - & 1 & - & - & 1 & - & - & - & - & 1 & - & - \\
- & - & 1 & - & 1 & 1 & 1 & - & - & - & - & - & - & 1 & - & - & - & - & 1 & - \\
1 & - & - & 1 & - & - & 1 & 1 & - & - & - & - & - & - & 1 & - & - & - & - & 1
    \end{pmatrix}\\
\end{align*}
\begin{align*}
    H_{20.2} =&
    \begin{pmatrix}
        1 & 1 & 1 & 1 & 1 & 1 & 1 & 1 & 1 & 1 & 1 & 1 & 1 & 1 & 1 & 1 & 1 & 1 & 1 & 1 \\
1 & 1 & 1 & 1 & 1 & 1 & 1 & 1 & 1 & 1 & - & - & - & - & - & - & - & - & - & - \\
1 & 1 & 1 & 1 & 1 & - & - & - & - & - & 1 & 1 & 1 & 1 & 1 & - & - & - & - & - \\
1 & 1 & 1 & 1 & - & 1 & - & - & - & - & 1 & - & - & - & - & 1 & 1 & 1 & 1 & - \\
1 & 1 & 1 & - & - & - & 1 & 1 & - & - & - & 1 & 1 & - & - & 1 & 1 & - & - & 1 \\
1 & 1 & - & 1 & - & - & 1 & 1 & - & - & - & - & - & 1 & 1 & - & - & 1 & 1 & 1 \\
1 & 1 & - & - & 1 & 1 & - & - & 1 & - & - & 1 & - & 1 & - & 1 & - & 1 & - & 1 \\
1 & 1 & - & - & 1 & 1 & - & - & - & 1 & - & - & 1 & - & 1 & - & 1 & - & 1 & 1 \\
1 & 1 & - & - & - & - & 1 & - & 1 & 1 & 1 & 1 & - & - & 1 & - & 1 & 1 & - & - \\
1 & 1 & - & - & - & - & - & 1 & 1 & 1 & 1 & - & 1 & 1 & - & 1 & - & - & 1 & - \\
1 & - & 1 & 1 & - & - & - & - & 1 & 1 & - & 1 & - & - & 1 & 1 & - & - & 1 & 1 \\
1 & - & 1 & - & 1 & - & 1 & - & 1 & - & 1 & - & 1 & - & - & - & - & 1 & 1 & 1 \\
1 & - & 1 & - & 1 & - & - & 1 & - & 1 & - & 1 & - & 1 & - & - & 1 & 1 & 1 & - \\
1 & - & 1 & - & - & 1 & 1 & - & - & 1 & - & - & 1 & 1 & 1 & 1 & - & 1 & - & - \\
1 & - & 1 & - & - & 1 & - & 1 & 1 & - & 1 & - & - & 1 & 1 & - & 1 & - & - & 1 \\
1 & - & - & 1 & 1 & - & 1 & - & - & 1 & 1 & - & - & 1 & - & 1 & 1 & - & - & 1 \\
1 & - & - & 1 & 1 & - & - & 1 & 1 & - & - & - & 1 & - & 1 & 1 & 1 & 1 & - & - \\
1 & - & - & 1 & - & 1 & 1 & - & 1 & - & - & 1 & 1 & 1 & - & - & 1 & - & 1 & - \\
1 & - & - & 1 & - & 1 & - & 1 & - & 1 & 1 & 1 & 1 & - & - & - & - & 1 & - & 1 \\
1 & - & - & - & 1 & 1 & 1 & 1 & - & - & 1 & 1 & - & - & 1 & 1 & - & - & 1 & -
    \end{pmatrix}
\end{align*}

\subsection{Order 36}\label{H36}
The matrix \(H'_{36}\) is the matrix \(H_{36}\) after switching a Hall set.
\begin{align*}
    H_{36} = & \left( \begin{smallmatrix}
        1 & - & - & - & - & - & - & 1 & 1 & 1 & 1 & - & - & 1 & 1 & 1 & - & 1 & - & 1 & 1 & - & 1 & 1 & - & 1 & - & 1 & 1 & 1 & - & - & 1 & 1 & 1 & 1 \\
- & 1 & - & - & - & - & - & - & 1 & 1 & 1 & 1 & 1 & - & 1 & 1 & 1 & - & 1 & - & 1 & 1 & - & 1 & 1 & - & 1 & - & 1 & 1 & 1 & - & - & 1 & 1 & 1 \\
- & - & 1 & - & - & - & 1 & - & - & 1 & 1 & 1 & - & 1 & - & 1 & 1 & 1 & 1 & 1 & - & 1 & 1 & - & 1 & 1 & - & 1 & - & 1 & 1 & 1 & - & - & 1 & 1 \\
- & - & - & 1 & - & - & 1 & 1 & - & - & 1 & 1 & 1 & - & 1 & - & 1 & 1 & - & 1 & 1 & - & 1 & 1 & 1 & 1 & 1 & - & 1 & - & 1 & 1 & 1 & - & - & 1 \\
- & - & - & - & 1 & - & 1 & 1 & 1 & - & - & 1 & 1 & 1 & - & 1 & - & 1 & 1 & - & 1 & 1 & - & 1 & - & 1 & 1 & 1 & - & 1 & 1 & 1 & 1 & 1 & - & - \\
- & - & - & - & - & 1 & 1 & 1 & 1 & 1 & - & - & 1 & 1 & 1 & - & 1 & - & 1 & 1 & - & 1 & 1 & - & 1 & - & 1 & 1 & 1 & - & - & 1 & 1 & 1 & 1 & - \\
- & - & 1 & 1 & 1 & 1 & 1 & - & - & - & - & - & - & 1 & 1 & 1 & 1 & - & - & 1 & 1 & 1 & - & 1 & - & 1 & 1 & - & 1 & 1 & - & 1 & - & 1 & 1 & 1 \\
1 & - & - & 1 & 1 & 1 & - & 1 & - & - & - & - & - & - & 1 & 1 & 1 & 1 & 1 & - & 1 & 1 & 1 & - & 1 & - & 1 & 1 & - & 1 & 1 & - & 1 & - & 1 & 1 \\
1 & 1 & - & - & 1 & 1 & - & - & 1 & - & - & - & 1 & - & - & 1 & 1 & 1 & - & 1 & - & 1 & 1 & 1 & 1 & 1 & - & 1 & 1 & - & 1 & 1 & - & 1 & - & 1 \\
1 & 1 & 1 & - & - & 1 & - & - & - & 1 & - & - & 1 & 1 & - & - & 1 & 1 & 1 & - & 1 & - & 1 & 1 & - & 1 & 1 & - & 1 & 1 & 1 & 1 & 1 & - & 1 & - \\
1 & 1 & 1 & 1 & - & - & - & - & - & - & 1 & - & 1 & 1 & 1 & - & - & 1 & 1 & 1 & - & 1 & - & 1 & 1 & - & 1 & 1 & - & 1 & - & 1 & 1 & 1 & - & 1 \\
- & 1 & 1 & 1 & 1 & - & - & - & - & - & - & 1 & 1 & 1 & 1 & 1 & - & - & 1 & 1 & 1 & - & 1 & - & 1 & 1 & - & 1 & 1 & - & 1 & - & 1 & 1 & 1 & - \\
- & 1 & - & 1 & 1 & 1 & - & - & 1 & 1 & 1 & 1 & 1 & - & - & - & - & - & - & 1 & 1 & 1 & 1 & - & - & 1 & 1 & 1 & - & 1 & - & 1 & 1 & - & 1 & 1 \\
1 & - & 1 & - & 1 & 1 & 1 & - & - & 1 & 1 & 1 & - & 1 & - & - & - & - & - & - & 1 & 1 & 1 & 1 & 1 & - & 1 & 1 & 1 & - & 1 & - & 1 & 1 & - & 1 \\
1 & 1 & - & 1 & - & 1 & 1 & 1 & - & - & 1 & 1 & - & - & 1 & - & - & - & 1 & - & - & 1 & 1 & 1 & - & 1 & - & 1 & 1 & 1 & 1 & 1 & - & 1 & 1 & - \\
1 & 1 & 1 & - & 1 & - & 1 & 1 & 1 & - & - & 1 & - & - & - & 1 & - & - & 1 & 1 & - & - & 1 & 1 & 1 & - & 1 & - & 1 & 1 & - & 1 & 1 & - & 1 & 1 \\
- & 1 & 1 & 1 & - & 1 & 1 & 1 & 1 & 1 & - & - & - & - & - & - & 1 & - & 1 & 1 & 1 & - & - & 1 & 1 & 1 & - & 1 & - & 1 & 1 & - & 1 & 1 & - & 1 \\
1 & - & 1 & 1 & 1 & - & - & 1 & 1 & 1 & 1 & - & - & - & - & - & - & 1 & 1 & 1 & 1 & 1 & - & - & 1 & 1 & 1 & - & 1 & - & 1 & 1 & - & 1 & 1 & - \\
- & 1 & 1 & - & 1 & 1 & - & 1 & - & 1 & 1 & 1 & - & - & 1 & 1 & 1 & 1 & 1 & - & - & - & - & - & - & 1 & 1 & 1 & 1 & - & - & 1 & 1 & 1 & - & 1 \\
1 & - & 1 & 1 & - & 1 & 1 & - & 1 & - & 1 & 1 & 1 & - & - & 1 & 1 & 1 & - & 1 & - & - & - & - & - & - & 1 & 1 & 1 & 1 & 1 & - & 1 & 1 & 1 & - \\
1 & 1 & - & 1 & 1 & - & 1 & 1 & - & 1 & - & 1 & 1 & 1 & - & - & 1 & 1 & - & - & 1 & - & - & - & 1 & - & - & 1 & 1 & 1 & - & 1 & - & 1 & 1 & 1 \\
- & 1 & 1 & - & 1 & 1 & 1 & 1 & 1 & - & 1 & - & 1 & 1 & 1 & - & - & 1 & - & - & - & 1 & - & - & 1 & 1 & - & - & 1 & 1 & 1 & - & 1 & - & 1 & 1 \\
1 & - & 1 & 1 & - & 1 & - & 1 & 1 & 1 & - & 1 & 1 & 1 & 1 & 1 & - & - & - & - & - & - & 1 & - & 1 & 1 & 1 & - & - & 1 & 1 & 1 & - & 1 & - & 1 \\
1 & 1 & - & 1 & 1 & - & 1 & - & 1 & 1 & 1 & - & - & 1 & 1 & 1 & 1 & - & - & - & - & - & - & 1 & 1 & 1 & 1 & 1 & - & - & 1 & 1 & 1 & - & 1 & - \\
- & 1 & 1 & 1 & - & 1 & - & 1 & 1 & - & 1 & 1 & - & 1 & - & 1 & 1 & 1 & - & - & 1 & 1 & 1 & 1 & 1 & - & - & - & - & - & - & 1 & 1 & 1 & 1 & - \\
1 & - & 1 & 1 & 1 & - & 1 & - & 1 & 1 & - & 1 & 1 & - & 1 & - & 1 & 1 & 1 & - & - & 1 & 1 & 1 & - & 1 & - & - & - & - & - & - & 1 & 1 & 1 & 1 \\
- & 1 & - & 1 & 1 & 1 & 1 & 1 & - & 1 & 1 & - & 1 & 1 & - & 1 & - & 1 & 1 & 1 & - & - & 1 & 1 & - & - & 1 & - & - & - & 1 & - & - & 1 & 1 & 1 \\
1 & - & 1 & - & 1 & 1 & - & 1 & 1 & - & 1 & 1 & 1 & 1 & 1 & - & 1 & - & 1 & 1 & 1 & - & - & 1 & - & - & - & 1 & - & - & 1 & 1 & - & - & 1 & 1 \\
1 & 1 & - & 1 & - & 1 & 1 & - & 1 & 1 & - & 1 & - & 1 & 1 & 1 & - & 1 & 1 & 1 & 1 & 1 & - & - & - & - & - & - & 1 & - & 1 & 1 & 1 & - & - & 1 \\
1 & 1 & 1 & - & 1 & - & 1 & 1 & - & 1 & 1 & - & 1 & - & 1 & 1 & 1 & - & - & 1 & 1 & 1 & 1 & - & - & - & - & - & - & 1 & 1 & 1 & 1 & 1 & - & - \\
- & 1 & 1 & 1 & 1 & - & - & 1 & 1 & 1 & - & 1 & - & 1 & 1 & - & 1 & 1 & - & 1 & - & 1 & 1 & 1 & - & - & 1 & 1 & 1 & 1 & 1 & - & - & - & - & - \\
- & - & 1 & 1 & 1 & 1 & 1 & - & 1 & 1 & 1 & - & 1 & - & 1 & 1 & - & 1 & 1 & - & 1 & - & 1 & 1 & 1 & - & - & 1 & 1 & 1 & - & 1 & - & - & - & - \\
1 & - & - & 1 & 1 & 1 & - & 1 & - & 1 & 1 & 1 & 1 & 1 & - & 1 & 1 & - & 1 & 1 & - & 1 & - & 1 & 1 & 1 & - & - & 1 & 1 & - & - & 1 & - & - & - \\
1 & 1 & - & - & 1 & 1 & 1 & - & 1 & - & 1 & 1 & - & 1 & 1 & - & 1 & 1 & 1 & 1 & 1 & - & 1 & - & 1 & 1 & 1 & - & - & 1 & - & - & - & 1 & - & - \\
1 & 1 & 1 & - & - & 1 & 1 & 1 & - & 1 & - & 1 & 1 & - & 1 & 1 & - & 1 & - & 1 & 1 & 1 & - & 1 & 1 & 1 & 1 & 1 & - & - & - & - & - & - & 1 & - \\
1 & 1 & 1 & 1 & - & - & 1 & 1 & 1 & - & 1 & - & 1 & 1 & - & 1 & 1 & - & 1 & - & 1 & 1 & 1 & - & - & 1 & 1 & 1 & 1 & - & - & - & - & - & - & 1
    \end{smallmatrix}\right)\\
    \\
    \\
    H'_{36} = &
    \left( \begin{smallmatrix}
        1 & - & - & - & - & - & - & - & 1 & 1 & - & - & - & 1 & - & 1 & - & - & - & 1 & 1 & - & 1 & 1 & - & - & - & 1 & - & 1 & - & - & - & 1 & 1 & - \\
- & 1 & - & - & - & - & - & - & 1 & 1 & 1 & 1 & 1 & - & 1 & 1 & 1 & - & 1 & - & 1 & 1 & - & 1 & 1 & - & 1 & - & 1 & 1 & 1 & - & - & 1 & 1 & 1 \\
- & - & 1 & - & - & - & 1 & - & - & 1 & 1 & 1 & - & 1 & - & 1 & 1 & 1 & 1 & 1 & - & 1 & 1 & - & 1 & 1 & - & 1 & - & 1 & 1 & 1 & - & - & 1 & 1 \\
- & - & - & 1 & - & - & 1 & - & - & - & - & 1 & 1 & - & - & - & 1 & - & - & 1 & 1 & - & 1 & 1 & 1 & - & 1 & - & - & - & 1 & 1 & - & - & - & - \\
- & - & - & - & 1 & - & 1 & 1 & 1 & - & - & 1 & 1 & 1 & - & 1 & - & 1 & 1 & - & 1 & 1 & - & 1 & - & 1 & 1 & 1 & - & 1 & 1 & 1 & 1 & 1 & - & - \\
- & - & - & - & - & 1 & 1 & 1 & 1 & 1 & - & - & 1 & 1 & 1 & - & 1 & - & 1 & 1 & - & 1 & 1 & - & 1 & - & 1 & 1 & 1 & - & - & 1 & 1 & 1 & 1 & - \\
- & - & 1 & 1 & 1 & 1 & 1 & - & - & - & - & - & - & 1 & 1 & 1 & 1 & - & - & 1 & 1 & 1 & - & 1 & - & 1 & 1 & - & 1 & 1 & - & 1 & - & 1 & 1 & 1 \\
- & - & - & - & 1 & 1 & - & 1 & - & - & - & - & - & - & 1 & 1 & 1 & 1 & - & - & 1 & - & 1 & - & 1 & - & 1 & 1 & - & 1 & 1 & - & 1 & - & 1 & 1 \\
1 & 1 & - & - & 1 & 1 & - & - & 1 & - & - & - & 1 & - & - & 1 & 1 & 1 & - & 1 & - & 1 & 1 & 1 & 1 & 1 & - & 1 & 1 & - & 1 & 1 & - & 1 & - & 1 \\
1 & 1 & 1 & - & - & 1 & - & - & - & 1 & - & - & 1 & 1 & - & - & 1 & 1 & 1 & - & 1 & - & 1 & 1 & - & 1 & 1 & - & 1 & 1 & 1 & 1 & 1 & - & 1 & - \\
- & 1 & 1 & - & - & - & - & - & - & - & 1 & - & 1 & 1 & 1 & - & - & 1 & - & 1 & - & - & - & 1 & 1 & - & 1 & 1 & - & 1 & - & 1 & 1 & 1 & - & 1 \\
- & 1 & 1 & 1 & 1 & - & - & - & - & - & - & 1 & 1 & 1 & 1 & 1 & - & - & 1 & 1 & 1 & - & 1 & - & 1 & 1 & - & 1 & 1 & - & 1 & - & 1 & 1 & 1 & - \\
- & 1 & - & 1 & 1 & 1 & - & - & 1 & 1 & 1 & 1 & 1 & - & - & - & - & - & - & 1 & 1 & 1 & 1 & - & - & 1 & 1 & 1 & - & 1 & - & 1 & 1 & - & 1 & 1 \\
1 & - & 1 & - & 1 & 1 & 1 & - & - & 1 & 1 & 1 & - & 1 & - & - & - & - & - & - & 1 & 1 & 1 & 1 & 1 & - & 1 & 1 & 1 & - & 1 & - & 1 & 1 & - & 1 \\
- & 1 & - & - & - & 1 & 1 & 1 & - & - & 1 & 1 & - & - & 1 & - & - & - & - & - & - & - & 1 & 1 & - & 1 & - & 1 & 1 & 1 & 1 & 1 & - & 1 & 1 & - \\
1 & 1 & 1 & - & 1 & - & 1 & 1 & 1 & - & - & 1 & - & - & - & 1 & - & - & 1 & 1 & - & - & 1 & 1 & 1 & - & 1 & - & 1 & 1 & - & 1 & 1 & - & 1 & 1 \\
- & 1 & 1 & 1 & - & 1 & 1 & 1 & 1 & 1 & - & - & - & - & - & - & 1 & - & 1 & 1 & 1 & - & - & 1 & 1 & 1 & - & 1 & - & 1 & 1 & - & 1 & 1 & - & 1 \\
- & - & 1 & - & 1 & - & - & 1 & 1 & 1 & 1 & - & - & - & - & - & - & 1 & - & 1 & 1 & - & - & - & 1 & 1 & 1 & - & 1 & - & 1 & 1 & - & 1 & 1 & - \\
- & 1 & 1 & - & 1 & 1 & - & - & - & 1 & - & 1 & - & - & - & 1 & 1 & - & 1 & - & - & - & - & - & - & - & 1 & 1 & - & - & - & 1 & - & 1 & - & - \\
1 & - & 1 & 1 & - & 1 & 1 & - & 1 & - & 1 & 1 & 1 & - & - & 1 & 1 & 1 & - & 1 & - & - & - & - & - & - & 1 & 1 & 1 & 1 & 1 & - & 1 & 1 & 1 & - \\
1 & 1 & - & 1 & 1 & - & 1 & 1 & - & 1 & - & 1 & 1 & 1 & - & - & 1 & 1 & - & - & 1 & - & - & - & 1 & - & - & 1 & 1 & 1 & - & 1 & - & 1 & 1 & 1 \\
- & 1 & 1 & - & 1 & 1 & 1 & - & 1 & - & - & - & 1 & 1 & - & - & - & - & - & - & - & 1 & - & - & 1 & - & - & - & - & 1 & 1 & - & - & - & 1 & - \\
1 & - & 1 & 1 & - & 1 & - & 1 & 1 & 1 & - & 1 & 1 & 1 & 1 & 1 & - & - & - & - & - & - & 1 & - & 1 & 1 & 1 & - & - & 1 & 1 & 1 & - & 1 & - & 1 \\
1 & 1 & - & 1 & 1 & - & 1 & - & 1 & 1 & 1 & - & - & 1 & 1 & 1 & 1 & - & - & - & - & - & - & 1 & 1 & 1 & 1 & 1 & - & - & 1 & 1 & 1 & - & 1 & - \\
- & 1 & 1 & 1 & - & 1 & - & 1 & 1 & - & 1 & 1 & - & 1 & - & 1 & 1 & 1 & - & - & 1 & 1 & 1 & 1 & 1 & - & - & - & - & - & - & 1 & 1 & 1 & 1 & - \\
- & - & 1 & - & 1 & - & 1 & - & 1 & 1 & - & 1 & 1 & - & 1 & - & 1 & 1 & - & - & - & - & 1 & 1 & - & 1 & - & - & - & - & - & - & 1 & 1 & 1 & 1 \\
- & 1 & - & 1 & 1 & 1 & 1 & 1 & - & 1 & 1 & - & 1 & 1 & - & 1 & - & 1 & 1 & 1 & - & - & 1 & 1 & - & - & 1 & - & - & - & 1 & - & - & 1 & 1 & 1 \\
1 & - & 1 & - & 1 & 1 & - & 1 & 1 & - & 1 & 1 & 1 & 1 & 1 & - & 1 & - & 1 & 1 & 1 & - & - & 1 & - & - & - & 1 & - & - & 1 & 1 & - & - & 1 & 1 \\
- & 1 & - & - & - & 1 & 1 & - & 1 & 1 & - & 1 & - & 1 & 1 & 1 & - & 1 & - & 1 & 1 & - & - & - & - & - & - & - & 1 & - & 1 & 1 & 1 & - & - & 1 \\
1 & 1 & 1 & - & 1 & - & 1 & 1 & - & 1 & 1 & - & 1 & - & 1 & 1 & 1 & - & - & 1 & 1 & 1 & 1 & - & - & - & - & - & - & 1 & 1 & 1 & 1 & 1 & - & - \\
- & 1 & 1 & 1 & 1 & - & - & 1 & 1 & 1 & - & 1 & - & 1 & 1 & - & 1 & 1 & - & 1 & - & 1 & 1 & 1 & - & - & 1 & 1 & 1 & 1 & 1 & - & - & - & - & - \\
- & - & 1 & 1 & 1 & 1 & 1 & - & 1 & 1 & 1 & - & 1 & - & 1 & 1 & - & 1 & 1 & - & 1 & - & 1 & 1 & 1 & - & - & 1 & 1 & 1 & - & 1 & - & - & - & - \\
- & - & - & - & 1 & 1 & - & 1 & - & 1 & 1 & 1 & 1 & 1 & - & 1 & 1 & - & - & 1 & - & - & - & 1 & 1 & 1 & - & - & 1 & 1 & - & - & 1 & - & - & - \\
1 & 1 & - & - & 1 & 1 & 1 & - & 1 & - & 1 & 1 & - & 1 & 1 & - & 1 & 1 & 1 & 1 & 1 & - & 1 & - & 1 & 1 & 1 & - & - & 1 & - & - & - & 1 & - & - \\
1 & 1 & 1 & - & - & 1 & 1 & 1 & - & 1 & - & 1 & 1 & - & 1 & 1 & - & 1 & - & 1 & 1 & 1 & - & 1 & 1 & 1 & 1 & 1 & - & - & - & - & - & - & 1 & - \\
- & 1 & 1 & - & - & - & 1 & 1 & 1 & - & 1 & - & 1 & 1 & - & 1 & 1 & - & - & - & 1 & - & 1 & - & - & 1 & 1 & 1 & 1 & - & - & - & - & - & - & 1
    \end{smallmatrix}\right)
\end{align*}

The matrix \(H_{36.n}\) is the normalization of \(H_{36}\), and \(H'_{36.n}\) is the matrix \(H_{36.n}\) after switching a Hall set.
\begin{align*}
H_{36.n} = &
    \left( \begin{smallmatrix}
        1 & 1 & 1 & 1 & 1 & 1 & 1 & 1 & 1 & 1 & 1 & 1 & 1 & 1 & 1 & 1 & 1 & 1 & 1 & 1 & 1 & 1 & 1 & 1 & 1 & 1 & 1 & 1 & 1 & 1 & 1 & 1 & 1 & 1 & 1 & 1 \\
1 & 1 & - & - & - & - & - & 1 & - & - & - & 1 & 1 & 1 & - & - & 1 & 1 & 1 & 1 & - & 1 & 1 & - & 1 & 1 & 1 & 1 & - & - & 1 & - & 1 & - & - & - \\
1 & - & 1 & - & - & - & 1 & 1 & 1 & - & - & 1 & - & - & 1 & - & 1 & - & 1 & - & 1 & 1 & - & 1 & 1 & - & - & - & 1 & - & 1 & 1 & 1 & 1 & - & - \\
1 & - & - & 1 & - & - & 1 & - & 1 & 1 & - & 1 & 1 & 1 & - & 1 & 1 & - & - & - & - & - & - & - & 1 & - & 1 & 1 & - & 1 & 1 & 1 & - & 1 & 1 & - \\
1 & - & - & - & 1 & - & 1 & - & - & 1 & 1 & 1 & 1 & - & 1 & - & - & - & 1 & 1 & - & 1 & 1 & - & - & - & 1 & - & 1 & - & 1 & 1 & - & - & 1 & 1 \\
1 & - & - & - & - & 1 & 1 & - & - & - & 1 & - & 1 & - & - & 1 & 1 & 1 & 1 & - & 1 & 1 & - & 1 & 1 & 1 & 1 & - & - & 1 & - & 1 & - & - & - & 1 \\
1 & - & 1 & 1 & 1 & 1 & 1 & 1 & 1 & 1 & 1 & - & - & - & - & - & 1 & 1 & - & - & - & 1 & 1 & - & - & - & 1 & 1 & - & - & - & 1 & 1 & - & - & - \\
1 & 1 & 1 & - & - & - & 1 & 1 & - & - & - & 1 & 1 & - & 1 & 1 & - & 1 & - & - & 1 & - & 1 & - & - & - & - & 1 & - & 1 & - & 1 & 1 & - & 1 & 1 \\
1 & - & 1 & 1 & - & - & 1 & - & 1 & - & - & 1 & - & - & - & 1 & - & 1 & 1 & 1 & - & - & 1 & 1 & - & 1 & 1 & 1 & 1 & - & - & - & - & 1 & - & 1 \\
1 & - & - & 1 & 1 & - & 1 & - & - & 1 & - & 1 & - & 1 & - & - & - & 1 & - & - & 1 & 1 & 1 & 1 & 1 & 1 & - & - & 1 & 1 & - & - & 1 & - & 1 & - \\
1 & - & - & - & 1 & 1 & 1 & - & - & - & 1 & 1 & - & 1 & 1 & - & 1 & 1 & - & 1 & - & - & - & 1 & - & - & - & 1 & - & 1 & 1 & - & 1 & 1 & - & 1 \\
1 & 1 & 1 & 1 & 1 & - & - & 1 & 1 & 1 & 1 & 1 & 1 & - & - & - & - & 1 & 1 & - & - & - & - & 1 & 1 & - & - & - & - & 1 & 1 & - & - & - & - & 1 \\
1 & 1 & - & 1 & 1 & 1 & - & 1 & - & - & - & 1 & 1 & 1 & 1 & 1 & - & 1 & - & - & - & 1 & - & 1 & - & - & 1 & - & 1 & - & - & 1 & - & 1 & - & - \\
1 & 1 & - & 1 & - & - & - & - & - & 1 & 1 & - & 1 & 1 & - & - & 1 & - & 1 & - & 1 & - & 1 & 1 & - & - & - & 1 & 1 & - & - & 1 & 1 & 1 & - & 1 \\
1 & - & 1 & - & 1 & - & - & 1 & - & - & 1 & - & 1 & - & 1 & - & 1 & - & - & - & - & - & 1 & 1 & 1 & 1 & 1 & 1 & 1 & 1 & - & - & - & 1 & 1 & - \\
1 & - & - & 1 & - & 1 & - & 1 & 1 & - & - & - & 1 & - & - & 1 & 1 & - & - & 1 & - & 1 & 1 & 1 & - & - & - & - & 1 & 1 & 1 & - & 1 & - & 1 & 1 \\
1 & 1 & 1 & 1 & - & 1 & 1 & - & - & - & 1 & - & - & 1 & 1 & 1 & 1 & 1 & 1 & - & - & - & 1 & - & 1 & - & - & - & 1 & - & 1 & - & - & - & 1 & - \\
1 & 1 & - & - & - & 1 & 1 & 1 & 1 & 1 & 1 & 1 & 1 & - & - & - & 1 & 1 & - & 1 & 1 & - & - & - & - & 1 & - & - & 1 & - & - & - & - & 1 & 1 & - \\
1 & 1 & 1 & - & 1 & 1 & - & - & 1 & - & - & 1 & - & 1 & - & - & 1 & - & 1 & 1 & 1 & - & 1 & 1 & - & - & 1 & - & - & 1 & - & 1 & - & - & 1 & - \\
1 & 1 & - & - & 1 & - & - & - & 1 & - & 1 & - & - & - & - & 1 & - & 1 & 1 & 1 & - & 1 & - & - & 1 & - & - & 1 & 1 & 1 & - & 1 & 1 & 1 & 1 & - \\
1 & - & 1 & - & - & 1 & - & 1 & - & 1 & - & - & - & 1 & - & - & - & 1 & 1 & - & 1 & 1 & - & - & - & - & 1 & 1 & 1 & 1 & 1 & - & - & 1 & 1 & 1 \\
1 & 1 & 1 & - & 1 & 1 & 1 & - & - & 1 & - & - & 1 & - & - & 1 & - & - & - & 1 & 1 & 1 & 1 & 1 & 1 & - & - & 1 & - & - & 1 & - & - & 1 & - & - \\
1 & 1 & - & - & 1 & - & 1 & 1 & 1 & 1 & - & - & - & 1 & 1 & 1 & 1 & - & 1 & - & - & 1 & 1 & - & - & 1 & - & - & - & 1 & - & - & - & 1 & - & 1 \\
1 & - & 1 & - & - & 1 & - & - & 1 & 1 & 1 & 1 & 1 & 1 & 1 & 1 & - & - & 1 & - & - & 1 & - & 1 & - & 1 & - & 1 & - & - & - & - & 1 & - & 1 & - \\
1 & 1 & 1 & 1 & - & 1 & - & - & - & 1 & - & 1 & - & - & 1 & - & 1 & - & - & 1 & - & 1 & - & - & 1 & 1 & - & 1 & 1 & 1 & - & 1 & - & - & - & 1 \\
1 & 1 & - & - & - & 1 & - & - & 1 & 1 & - & - & - & - & 1 & - & - & 1 & - & - & - & - & 1 & 1 & 1 & 1 & 1 & - & - & - & 1 & 1 & 1 & 1 & 1 & 1 \\
1 & 1 & - & 1 & 1 & 1 & 1 & - & 1 & - & - & - & 1 & - & 1 & - & - & - & 1 & - & 1 & - & - & - & - & 1 & 1 & 1 & 1 & 1 & 1 & - & 1 & - & - & - \\
1 & 1 & - & 1 & - & - & 1 & 1 & 1 & - & 1 & - & - & 1 & 1 & - & - & - & - & 1 & 1 & 1 & - & 1 & 1 & - & 1 & 1 & - & - & - & - & - & - & 1 & 1 \\
1 & - & 1 & - & 1 & - & - & - & 1 & 1 & - & - & 1 & 1 & 1 & 1 & 1 & 1 & - & 1 & 1 & - & - & - & 1 & - & 1 & - & 1 & - & - & - & 1 & - & - & 1 \\
1 & - & - & 1 & - & 1 & - & 1 & - & 1 & 1 & 1 & - & - & 1 & 1 & - & - & 1 & 1 & 1 & - & 1 & - & 1 & - & 1 & - & - & 1 & - & - & 1 & 1 & - & - \\
1 & 1 & 1 & 1 & 1 & - & - & - & - & - & 1 & 1 & - & - & - & 1 & 1 & - & - & - & 1 & 1 & - & - & - & 1 & 1 & - & - & - & 1 & - & 1 & 1 & 1 & 1 \\
1 & - & 1 & 1 & 1 & 1 & 1 & 1 & - & - & - & - & 1 & 1 & - & - & - & - & 1 & 1 & - & - & - & - & 1 & 1 & - & - & - & - & - & 1 & 1 & 1 & 1 & 1 \\
1 & 1 & 1 & - & - & - & 1 & 1 & - & 1 & 1 & - & - & 1 & - & 1 & - & - & - & 1 & - & - & - & 1 & - & 1 & 1 & - & 1 & 1 & 1 & 1 & 1 & - & - & - \\
1 & - & 1 & 1 & - & - & - & - & 1 & - & 1 & - & 1 & 1 & 1 & - & - & 1 & - & 1 & 1 & 1 & 1 & - & - & 1 & - & - & - & 1 & 1 & 1 & - & 1 & - & - \\
1 & - & - & 1 & 1 & - & - & 1 & - & 1 & - & - & - & - & 1 & 1 & 1 & 1 & 1 & 1 & 1 & - & - & 1 & - & 1 & - & 1 & - & - & 1 & 1 & - & - & 1 & - \\
1 & - & - & - & 1 & 1 & - & 1 & 1 & - & 1 & 1 & - & 1 & - & 1 & - & - & - & - & 1 & - & 1 & - & 1 & 1 & - & 1 & 1 & - & 1 & 1 & - & - & - & 1
    \end{smallmatrix}\right)\\
    &\\
    &\\
    H'_{36.n} = & \left( \begin{smallmatrix}
        1 & 1 & 1 & 1 & 1 & 1 & 1 & 1 & 1 & 1 & 1 & 1 & 1 & 1 & 1 & 1 & 1 & 1 & 1 & 1 & 1 & 1 & 1 & 1 & 1 & 1 & 1 & 1 & 1 & 1 & 1 & 1 & 1 & 1 & 1 & 1 \\
1 & 1 & - & - & - & - & - & 1 & 1 & - & - & - & - & 1 & - & 1 & 1 & 1 & 1 & 1 & - & 1 & 1 & - & 1 & 1 & - & 1 & - & 1 & - & - & 1 & 1 & - & - \\
1 & - & 1 & - & - & - & 1 & 1 & 1 & - & - & 1 & - & - & 1 & - & 1 & - & 1 & - & 1 & 1 & - & 1 & 1 & - & - & - & 1 & - & 1 & 1 & 1 & 1 & - & - \\
1 & - & - & 1 & - & - & 1 & - & 1 & 1 & - & 1 & 1 & 1 & - & 1 & 1 & - & - & - & - & - & - & - & 1 & - & 1 & 1 & - & 1 & 1 & 1 & - & 1 & 1 & - \\
1 & - & - & - & 1 & - & 1 & - & 1 & 1 & 1 & - & - & - & 1 & 1 & - & - & 1 & 1 & - & 1 & 1 & - & - & - & - & - & 1 & 1 & - & 1 & - & 1 & 1 & 1 \\
1 & - & - & - & - & 1 & 1 & - & - & - & 1 & - & 1 & - & - & 1 & 1 & 1 & 1 & - & 1 & 1 & - & 1 & 1 & 1 & 1 & - & - & 1 & - & 1 & - & - & - & 1 \\
1 & - & 1 & 1 & 1 & 1 & 1 & 1 & 1 & 1 & 1 & - & - & - & - & - & 1 & 1 & - & - & - & 1 & 1 & - & - & - & 1 & 1 & - & - & - & 1 & 1 & - & - & - \\
1 & 1 & 1 & - & - & - & 1 & 1 & - & - & - & 1 & 1 & - & 1 & 1 & - & 1 & - & - & 1 & - & 1 & - & - & - & - & 1 & - & 1 & - & 1 & 1 & - & 1 & 1 \\
1 & 1 & 1 & 1 & 1 & - & 1 & - & 1 & - & - & 1 & - & - & - & 1 & - & 1 & 1 & - & - & - & - & 1 & - & 1 & 1 & 1 & 1 & - & - & - & - & 1 & - & 1 \\
1 & - & - & 1 & 1 & - & 1 & - & - & 1 & - & 1 & - & 1 & - & - & - & 1 & - & - & 1 & 1 & 1 & 1 & 1 & 1 & - & - & 1 & 1 & - & - & 1 & - & 1 & - \\
1 & - & - & - & 1 & 1 & 1 & - & - & - & 1 & 1 & - & 1 & 1 & - & 1 & 1 & - & 1 & - & - & - & 1 & - & - & - & 1 & - & 1 & 1 & - & 1 & 1 & - & 1 \\
1 & - & 1 & 1 & - & - & - & 1 & 1 & 1 & 1 & 1 & 1 & - & - & - & - & 1 & 1 & 1 & - & - & 1 & 1 & 1 & - & - & - & - & 1 & 1 & - & - & - & - & 1 \\
1 & - & - & 1 & - & 1 & - & 1 & - & - & - & 1 & 1 & 1 & 1 & 1 & - & 1 & - & 1 & - & 1 & 1 & 1 & - & - & 1 & - & 1 & - & - & 1 & - & 1 & - & - \\
1 & 1 & - & 1 & - & - & - & - & - & 1 & 1 & - & 1 & 1 & - & - & 1 & - & 1 & - & 1 & - & 1 & 1 & - & - & - & 1 & 1 & - & - & 1 & 1 & 1 & - & 1 \\
1 & - & 1 & - & 1 & - & - & 1 & - & - & 1 & - & 1 & - & 1 & - & 1 & - & - & - & - & - & 1 & 1 & 1 & 1 & 1 & 1 & 1 & 1 & - & - & - & 1 & 1 & - \\
1 & 1 & - & 1 & 1 & 1 & - & 1 & 1 & - & - & - & 1 & - & - & 1 & 1 & - & - & - & - & 1 & - & 1 & - & - & - & - & 1 & 1 & 1 & - & 1 & - & 1 & 1 \\
1 & 1 & 1 & 1 & - & 1 & 1 & - & - & - & 1 & - & - & 1 & 1 & 1 & 1 & 1 & 1 & - & - & - & 1 & - & 1 & - & - & - & 1 & - & 1 & - & - & - & 1 & - \\
1 & 1 & - & - & - & 1 & 1 & 1 & 1 & 1 & 1 & 1 & 1 & - & - & - & 1 & 1 & - & 1 & 1 & - & - & - & - & 1 & - & - & 1 & - & - & - & - & 1 & 1 & - \\
1 & 1 & 1 & - & 1 & 1 & - & - & 1 & - & - & 1 & - & 1 & - & - & 1 & - & 1 & 1 & 1 & - & 1 & 1 & - & - & 1 & - & - & 1 & - & 1 & - & - & 1 & - \\
1 & 1 & - & - & 1 & - & - & - & - & - & 1 & 1 & 1 & - & - & - & - & 1 & 1 & 1 & - & 1 & - & - & 1 & - & 1 & 1 & 1 & - & 1 & 1 & 1 & - & 1 & - \\
1 & - & 1 & - & - & 1 & - & 1 & - & 1 & - & - & - & 1 & - & - & - & 1 & 1 & - & 1 & 1 & - & - & - & - & 1 & 1 & 1 & 1 & 1 & - & - & 1 & 1 & 1 \\
1 & 1 & 1 & - & 1 & 1 & 1 & - & - & 1 & - & - & 1 & - & - & 1 & - & - & - & 1 & 1 & 1 & 1 & 1 & 1 & - & - & 1 & - & - & 1 & - & - & 1 & - & - \\
1 & 1 & - & - & 1 & - & 1 & 1 & - & 1 & - & 1 & 1 & 1 & 1 & - & 1 & - & 1 & - & - & 1 & 1 & - & - & 1 & 1 & - & - & - & 1 & - & - & - & - & 1 \\
1 & - & 1 & - & - & 1 & - & - & 1 & 1 & 1 & 1 & 1 & 1 & 1 & 1 & - & - & 1 & - & - & 1 & - & 1 & - & 1 & - & 1 & - & - & - & - & 1 & - & 1 & - \\
1 & 1 & 1 & 1 & - & 1 & - & - & - & 1 & - & 1 & - & - & 1 & - & 1 & - & - & 1 & - & 1 & - & - & 1 & 1 & - & 1 & 1 & 1 & - & 1 & - & - & - & 1 \\
1 & 1 & - & - & - & 1 & - & - & 1 & 1 & - & - & - & - & 1 & - & - & 1 & - & - & - & - & 1 & 1 & 1 & 1 & 1 & - & - & - & 1 & 1 & 1 & 1 & 1 & 1 \\
1 & - & - & 1 & - & 1 & 1 & - & 1 & - & - & - & 1 & - & 1 & - & - & - & 1 & 1 & 1 & - & 1 & - & - & 1 & 1 & 1 & 1 & 1 & 1 & - & 1 & - & - & - \\
1 & 1 & - & 1 & - & - & 1 & 1 & 1 & - & 1 & - & - & 1 & 1 & - & - & - & - & 1 & 1 & 1 & - & 1 & 1 & - & 1 & 1 & - & - & - & - & - & - & 1 & 1 \\
1 & - & 1 & - & 1 & - & - & - & 1 & 1 & - & - & 1 & 1 & 1 & 1 & 1 & 1 & - & 1 & 1 & - & - & - & 1 & - & 1 & - & 1 & - & - & - & 1 & - & - & 1 \\
1 & 1 & - & 1 & 1 & 1 & - & 1 & - & 1 & 1 & 1 & - & - & 1 & 1 & - & - & 1 & - & 1 & - & - & - & 1 & - & 1 & - & - & 1 & - & - & 1 & 1 & - & - \\
1 & - & 1 & 1 & - & - & - & - & - & - & 1 & 1 & - & - & - & 1 & 1 & - & - & 1 & 1 & 1 & 1 & - & - & 1 & 1 & - & - & - & 1 & - & 1 & 1 & 1 & 1 \\
1 & - & 1 & 1 & 1 & 1 & 1 & 1 & - & - & - & - & 1 & 1 & - & - & - & - & 1 & 1 & - & - & - & - & 1 & 1 & - & - & - & - & - & 1 & 1 & 1 & 1 & 1 \\
1 & 1 & 1 & - & - & - & 1 & 1 & - & 1 & 1 & - & - & 1 & - & 1 & - & - & - & 1 & - & - & - & 1 & - & 1 & 1 & - & 1 & 1 & 1 & 1 & 1 & - & - & - \\
1 & 1 & 1 & 1 & 1 & - & - & - & 1 & - & 1 & - & 1 & 1 & 1 & - & - & 1 & - & - & 1 & 1 & - & - & - & 1 & - & - & - & 1 & 1 & 1 & - & 1 & - & - \\
1 & - & - & 1 & 1 & - & - & 1 & - & 1 & - & - & - & - & 1 & 1 & 1 & 1 & 1 & 1 & 1 & - & - & 1 & - & 1 & - & 1 & - & - & 1 & 1 & - & - & 1 & - \\
1 & - & - & - & 1 & 1 & - & 1 & 1 & - & 1 & 1 & - & 1 & - & 1 & - & - & - & - & 1 & - & 1 & - & 1 & 1 & - & 1 & 1 & - & 1 & 1 & - & - & - & 1
    \end{smallmatrix}\right)
\end{align*}

\end{document}